\documentclass[12pt]{amsart}
\pdfoutput=1

\usepackage[nobysame,abbrev,alphabetic]{amsrefs}
\usepackage{amssymb}
\usepackage[all]{xy}
\usepackage{stmaryrd}
\usepackage{hyperref}
\hypersetup{citecolor=blue}

\DeclareMathOperator{\cd}{cd}

\DeclareMathOperator{\Hom}{Hom}
\DeclareMathOperator{\id}{id}

\DeclareMathOperator{\res}{res}

\newcommand{\discup}{\ \ensuremath{\mathaccent\cdot\cup}}

\newcommand{\gr}{\mathrm{gr}}
\newcommand{\nek}{,\ldots,}

\newcommand{\inv}{^{-1}}
\newcommand{\isom}{\cong}
\newcommand{\lex}{\mathrm{lex}}

\newcommand{\pr}{\mathrm{pr}}

\newcommand{\Shu}{\mathrm{Sh}}
\newcommand{\trg}{\mathrm{trg}}
\newcommand{\tensor}{\otimes}

\newtheorem{thm}{Theorem}[section]
\newtheorem{cor}[thm]{Corollary}
\newtheorem{lem}[thm]{Lemma}
\newtheorem{prop}[thm]{Proposition}
\newtheorem{defin}[thm]{Definition}
\newtheorem{exam}[thm]{Example}
\newtheorem{examples}[thm]{Examples}
\newtheorem{rem}[thm]{Remark}

\newtheorem*{thmA}{Theorem A}
\newtheorem*{thmB}{Theorem B}
\swapnumbers

\numberwithin{equation}{section}

\newcommand{\alp}{\alpha}
\newcommand{\gam}{\gamma}

\newcommand{\del}{\delta}

\newcommand{\eps}{\epsilon}

\newcommand{\Lam}{\Lambda}
\newcommand{\sig}{\sigma}

\newcommand{\dbF}{\mathbb{F}}

\newcommand{\dbQ}{\mathbb{Q}}

\newcommand{\dbU}{\mathbb{U}}

\newcommand{\dbZ}{\mathbb{Z}}

\newcommand{\grz}{\mathfrak{z}}

\newcommand{\calM}{\mathcal{M}}
\newcommand{\calN}{\mathcal{N}}
\newcommand{\calR}{\mathcal{R}}

\DeclareFontFamily{U}{wncy}{}
\DeclareFontShape{U}{wncy}{m}{n}{<->wncyr10}{}
\DeclareSymbolFont{mcy}{U}{wncy}{m}{n}
\DeclareMathSymbol{\Sha}{\mathord}{mcy}{"58}
\DeclareMathSymbol{\sha}{\mathord}{mcy}{"78}

\begin{document}

\title{Mild Pro-$p$ Groups and Ordered Monoids}

\author{Ido Efrat}
\address{Earl Katz Family Chair in Pure Mathematics\\
Department of Mathematics\\
Ben-Gurion University of the Negev\\
P.O.\ Box 653, Be'er-Sheva 84105\\
Israel} \email{efrat@bgu.ac.il}

\thanks{This work was supported by the Israel Science Foundation (grant No.\ 569/21). }



\setcounter{tocdepth}{1}

\maketitle

\begin{abstract}
We prove a criterion for the mildness of a finitely presented pro-$p$ group $G$.
It implies as a special case a cohomological mildness criterion via Massey products, generalizing results due to Schmidt and G\"artner.
It subsumes Labute's non-singular circuit criterion.
We further show connections with the triangle condition for the mildness of pro-$p$ right-angled Artin groups, due to Quadrelli, Snopce and Vannacci.  
\end{abstract}

\tableofcontents

\section{Introduction}
\subsection{Mild presentations}
We fix a prime $p$ and consider a pro-$p$ group $G$ with a minimal finite presentation.
Thus $G=F/R$, where $F$ is a free pro-$p$ group on a basis $X=\{x_1\nek x_d\}$ with $1\leq d<\infty$ and $R$ is a closed normal subgroup of $F$ generated by elements $r_1\nek r_m$ of $F$, with $1\leq m<\infty$ minimal.
In some fortunate cases, general structural properties of the presentation yield significant insights into the structure of $G$.
As a very simple example, the condition $m=0$ means that $G$ is a free pro-$p$ group, which is equivalent to having cohomological dimension $\cd(G)$ at most $1$.
An additional and much deeper example is the celebrated Golod--Shafarevich theorem:
If $d^2\geq 4m$, then $G$ is infinite.

In this paper we study another such situation, namely that of \textsl{mild} pro-$p$ groups $G$.
While not free pro-$p$, such groups admit a minimal presentation in which the initial forms of the relations are ``free in the maximal possible way''.
This leads to many attractive consequences for $G$, for example, $\cd(G)=2$.
Alternatively, $G$ is mild if and only if the Golod--Shafarevich inequality, in its graded-algebra version, becomes an equality for a certain graded algebra naturally associated to a presentation of $G$.
We refer to \S\ref{section on mild groups} for the precise definition, which is due to Labute and Anick.
We mention at this point that the associated graded algebra depends not only on the presentation of $G$ but also on fixing a filtration of $F$, typical examples being the $p$-Zassenhaus filtration $F_{(n)}$, $n\geq1$, or the lower $p$-central filtration $F^{(n)}$, $n\geq1$ (see \S\ref{section on the Zassenhaus filtration}).

Recent work has revealed a rich web of connections between mild groups and major problems across algebra and arithmetic:
In group theory (the Gildenhuys--Serre conjecture on one-relator pro-$p$ groups \cite{Gildenhuys68}, \cite{Gartner15}*{\S5}),
in algebraic number theory (the structure of restricted ramification Galois groups of number fields \cite{Labute06}, \cite{Schmidt06}, \cite{Schmidt07}, \cite{Vogel05}, \cite{Maire14}),
in Galois cohomology (higher Massey products and the Koszulity Conjecture \cite{MinacPasiniQuadrelliTan22}), and
in arithmetic geometry (the tame Fontaine--Mazur Conjecture \cite{Labute14}, \cite{Maire07}).

Verifying ``by hand'' that a given finitely presented pro-$p$ group is mild is generally difficult.
Several sufficient conditions for mildness, to be discussed below, were proved in works by Anick \cite{Anick82}, Labute \cite{Labute06}, Schmidt \cite{Schmidt10}, \cite{Schmidt07}, G\"artner \cite{Gartner15}, 
Labute--Min\'a\v c \cite{LabuteMinac11}, Quadrelli--Snopce--Vannacci \cite{QuadrelliSnopceVannacci22} and others.  
See also works by Forr\'e \cite{Forre11} and Maire \cite{Maire14}.
These criteria were developed in diverse contexts, including Lie algebras, Galois cohomology, homological algebra, combinatorial group theory and graph theory.

\subsection{Main results}
The aim of the current paper is to give a unified approach to mildness, based on connections between the combinatorics of words and profinite group cohomology.
In particular, we prove a general mildness criterion that subsumes several of the main known criteria as special cases, and in addition implies mildness in a variety of new situations.

Fundamental to our approach is the \textsl{Massey product} in profinite group cohomology.
We write $H^i(G)=H^i(G,\dbF_p)$ for the $i$th cohomology group of $G$ with the trivial action on $\dbF_p$.
The $n$-fold Massey product is in general a multi-valued map
\[
\langle \cdot\nek\cdot\rangle\colon H^1(G)^n\to H^2(G).
\]
However, in the situations considered here it is a single valued and multi-linear map.
When $n=2$ it consists of the cup product.
See \S\ref{subsection on Massey products} and \S\ref{subsection on the p-Zassenhaus invariant}.

To state our general mildness criterion, let  $X^*$ be the free monoid on the basis $X$ of $F$.
We consider the elements of $X^*$ as words in the alphabet $X$.
We also take $\chi_x$, $x\in X$, to be the linear basis of $H^1(G)=\Hom_{\rm cont}(G,\dbZ/p)$ which is dual to $X$, i.e., $\chi_x(x)=1$ and $\chi_x(x')=0$ for $x\neq x'$.
For every $w=(x_1\cdots x_n)\in X^n$ we define in a canonical way a Massey product element $\alp_w\in\langle\chi_{x_1}\nek\chi_{x_n}\rangle\subseteq H^2(G)$  (See \S\ref{subsection on Magnus respresentations for words}). 

Our main criterion for mildness is based on finding \textsl{ordered monoid structures} on $X^*$ for which a special set of cohomology elements $\alp_w$ satisfies a maximality property and simultaneously generates the whole group $H^2(G)$:

\begin{thmA}
\label{Theorem A}
Let $\leq$ be a total order on $X^*$ such that $(X^*,\leq)$ is an ordered monoid.
Suppose that $r_1\nek r_m\in F_{(n)}\setminus F_{(n+1)}$.
Let $B\subseteq X^n$ satisfy:
\begin{enumerate}
\item[(a)]
No proper prefix of a word in $B$ is a proper suffix of a word in $B$.
\item[(b)]
If $w_1\in B$ and $w_2\in X^n$ satisfy $\alp_{w_1},\alp_{w_2}\neq0$ and $w_1\leq w_2$, then also $w_2\in B$.
\item[(c)]
The cohomology elements $\alp_w$, $w\in B$, generate $H^2(G)$.
\end{enumerate}
Then $G=F/R$ is mild with respect to the $p$-Zassenhaus filtration. 
\end{thmA}
See Theorem \ref{main theorem on mildness} for a more general statement.

The applications of Theorem A rely on the construction of total orders on the free monoid $X^*$, specifically tailored to the type of presentation under consideration.
As a primary application, we establish the following generalization of the mildness criteria due to Schmidt and Gärtner (Theorem \ref{theorem on Massey products}):

\begin{thmB}
\label{Theorem B}
Let $n$ be maximal such that $G$ has a presentation as above with $R\subseteq F_{(n)}$.
Let $s,k_0,k_1\nek k_s$ be positive integers with $n=k_0+k_1+\cdots+k_s$.
Consider nonzero $\dbF_p$-linear subspaces $U_0\nek U_s$ of $H^1(G)$ such that $H^1(G)=U_0\oplus\cdots\oplus U_s$.
Assume that:
\begin{enumerate}
\item[(a)]
One has $\langle\varphi_1\nek\varphi_n\rangle=0$ whenever $\varphi_1\nek\varphi_n\in H^1(G)$ and more than $k_j$ of them belong to $U_j$ for some $1\leq j\leq s$.
\item[(b)]
The group $H^2(G)$ is generated by the image of
\[
\langle\cdot\nek\cdot\rangle\colon U_0^{k_0}\times\cdots\times U_s^{k_s}\to H^2(G).
\]
\end{enumerate}
Then $G$ has a mild presentation with respect to the $p$-Zassenhaus filtration.
\end{thmB}

Specifically, the case $s=1$, $k_0=k_1=1$, $n=2$ of Theorem B recovers results from \cite{Schmidt10,Schmidt07} for $p\neq2$, complemented in \cite{LabuteMinac11} for $p=2$.
The case $s=1$, $n\geq2$ is due to G\"artner \cite{Gartner15}.
The general case is based on a subtle construction of a total order on $X^*$ with special properties as required in Theorem A.

\subsection{Connections to other criteria}
A mildness criterion of a different flavor was given by Labute \cite{Labute06}, motivated by the structure of  
the maximal pro-$p$ Galois group $G_{\dbQ,S}(p)$ of $\dbQ$ unramified outside a finite set $S$ of primes.
Under some assumptions on $S$, the group has a presentation of a particularly neat form, known as a \textsl{Koch presentation}.
Namely, $1\leq m\leq d$ and the defining relations $r_1\nek r_m$ in $R$ satisfy 
\[
r_j=x_j^{pa_j}\prod_{{k=1}\atop{k\neq j}}^d[x_j,x_k]^{a_{jk}}\pmod{F_{(3)}}
\]
with $0\leq a_j,a_{jk}\leq p-1$.
The \textsl{linking diagram} associated with the presentation is the labelled directed graph with vertices $x_1\nek x_d$ and edges with label $a_{ij}$ from $x_i$ to $x_j$ when $a_{ij}\neq0$.
Labute proves that, if $p\neq2$, $d\geq4$ and there is a circuit in the graph passing through all vertices and satisfying certain technical conditions on the labels along it, then $G$ is mild with respect to the lower $p$-central filtration.
The case $p=2$ was completed in \cite{LabuteMinac11}.
As we show in \S\ref{section on the non-singular circuit criteria}, this criterion is also a special case of Theorem A, when one chooses the total order on $X^*$ and the set $B$ appropriately. 

A mildness criterion of yet another kind is given by Quadrelli, Snopce and Vannacci \cite{QuadrelliSnopceVannacci22}, building on homological results of Weigel \cite{Weigel15}.
Here to a (combinatorial) graph $\Gamma$ one associates the \textsl{pro-$p$ right-angled Artin group} $G_\Gamma$, with vertices of $\Gamma$ as generators and relations $[v,v']=1$ for every edge $(v,v')$
(the authors in fact allow certain relations of the form $[v,v']v^{p\alpha}(v')^{p\alpha'}$ for $\alpha,\alpha'\in\dbZ_p$).
They prove that $G_\Gamma$ is mild if and only if $\Gamma$ contains no triangles as induced subgraphs.
In \S\ref{section on pro-p RAAGs} we show that Theorem A implies the mildness of $G_\Gamma$ provided that $\Gamma$ is bipartite.
Note that in this case $\Gamma$ indeed contains no triangles.

\subsection{The structure of the paper}
In \S\ref{section on ordered monoids} we study ordered structures on monoids and present a lexicographic construction of such a structure.
This construction is developed further in \S\ref{section on partictions and mildness}, where it serves as a key step toward the proof of Theorem B.  
In \S\ref{section on Magnus representations} and \S\ref{section on unitriangular cohomology elements} we construct the Massey product elements $\alpha_w$ using pro-$p$ Magnus theory and explain their connection to relations via the transgression pairing.  
In \S\ref{section on coefficient matrices} we give conditions equivalent to condition (c) of Theorem A, formulated in terms of matrices of Magnus coefficients.  
In \S\ref{section on mild groups} we recall the definition of mild groups, presenting it in a uniform way that applies to both the $p$-Zassenhaus filtration and the lower $p$-central filtration.  Building on these preparations, in \S\ref{section on the mildness criterion} we state and prove a slightly stronger version of Theorem A.  
Combining this result with the lexicographic ordered structure constructed earlier on the free monoid $X^*$, we obtain in \S\ref{section on partictions and mildness} a word-combinatorial criterion for mildness (Theorem \ref{partition criterion}).  
After recalling some basic facts about the standard filtrations in \S\ref{section on the Zassenhaus filtration}, we translate Theorem \ref{partition criterion} into the language of Massey products in \S\ref{section on Massey product criteria}, yielding Theorem B.  
These results go beyond the earlier work of Schmidt and G\"artner, providing a wealth of new cohomological mildness criteria.
Here, shuffle relations for Massey products (as found in \cite{Efrat24} and \cite{Vogel04}) allow us to substantially weaken the assumptions on the vanishing of Massey products required by these criteria.  
Finally, in \S\ref{section on the non-singular circuit criteria} and \S\ref{section on pro-p RAAGs} we show that Labute’s linking diagram criterion, as well as the triangle criterion of \cite{QuadrelliSnopceVannacci22} for bipartite graphs, are special cases of Theorem A.

\section{Ordered monoids}
\label{section on ordered monoids}
\subsection{Generalities on monoids}
\label{subsection on generalities on monoids}
Let $\calM$ be a monoid with unit element $1_\calM$.
We denote its multiplication map by $(w,u)\mapsto wu$.
Given $w,w'\in \calM$, we say that $w$ is a \textsl{left} (resp., \textsl{right}, \textsl{middle}) factor of $w'$ if $w'=wv$ (resp., $w'=uw$, $w'=uwv$) for some $u,v\in\calM$.
The monoid $\calM$ is \textsl{cancellative} if for every $w,w',u\in\calM$, any of the equalities $wu=w'u$, $uw=uw'$ implies that $w=w'$.

\begin{defin}
\rm
An \textsl{ordered monoid} $(\calM,\leq)$ consists of a monoid $\calM$ and a total order $\leq$ on the underlying set of $\calM$ such that:
\begin{enumerate}
\item[(i)]
$1_\calM$ is minimal in $\calM$ with respect to $\leq$;
\item[(ii)]
The product map $\calM\times \calM\to \calM$ is order-preserving, where we partially order $\calM\times \calM$ by the product order.
\end{enumerate}
\end{defin}

\begin{rem}
\label{remarks on ordered monoids}
\rm
(1) \quad
For $w,w',u,u'$ in a cancellative ordered monoid such that $w\leq w'$ and $u\leq u'$, with at least one of the inequalities being strict, one has $wu<w'u'$.

\medskip

(2) \quad
Condition (ii) is equivalent to: 
\begin{enumerate}
\item[(ii')]
For every words $u,v,w,w'\in \calM$ with $w\leq w'$ one has $uwv\leq uw'v$.
\end{enumerate}

\medskip
(3) \quad
In the context of the monoid of associative monomials in a given set of variables, a total order satisfying (i) and (ii') is called in \cite{Gartner15}*{Def.\ 3.3} a \textsl{multiplicative order}. 
\end{rem}

\subsection{A lexicographic construction}
\label{subsection on a lexicographic construction}

The following construction of ordered monoids will be fundamental for the proof of the Massey product criterion for mildness in \S\ref{section on partictions and mildness}--\S\ref{section on Massey product criteria}.

\begin{lem}
\label{g}
Let $t\geq1$ and let $(\calM_i,\leq_i)$, $i=1,2\nek t$, be cancellative ordered monoids.
Let $\leq_\lex$ be the lexicographic order on $\prod_{i=1}^t\calM_i$ induced by the orders $\leq_i$.
Consider a monoid $\calN$ and maps $g_i\colon \calN\to\calM_i$, $i=1,2\nek t$, such that $g_i(1_\calN)=1_{\calM_i}$ for every $i$ and such that the map 
\[
g=(g_1\nek g_t)\colon\calN\to\prod_{i=1}^t\calM_i
\]
is injective.
Assume further that there are maps
\[
h_i\colon \Bigl(\prod_{j=1}^{i-1}\calM_j\Bigr)^2\to\calM_i,\quad i=1,2\nek t,
\]
where $h_1=1_{\calM_1}$, such that for every $1\leq i\leq t$ and $w,u\in\calN$ we have
\begin{equation}
\label{connection between maps}
g_i(wu)=h_i\bigl(g_1(w)\nek g_{i-1}(w),g_1(u)\nek g_{i-1}(u)\bigr)g_i(w)g_i(u).
\end{equation}
Then $(\calN,\leq_g)$ is an ordered monoid, where $\leq_g$ is given by 
\[
w\leq_gw'\quad\Longleftrightarrow\quad g(w)\leq_\lex g(w').
\] 
\end{lem}
\begin{proof}
It follows from the injectivity of $g$ that $\leq_g$ is a total order on $\calN$.

We have $g(1_\calN)=(1_{\calM_1}\nek 1_{\calM_t})$, which is minimal in $(\prod_{i=1}^t\calM_i,\leq_{\rm lex})$.
Hence $1_\calN$ is minimal in $(\calN,\leq_g)$.

It remains to show that for every $w,w',u,u'\in \calN$ one has:
\[
g(w)\leq_\lex g(w'), \quad g(u)\leq_\lex g(u')\ \Rightarrow\ g(wu)\leq_\lex g(w'u').
\]
We take $0\leq i\leq t$ maximal such that
\begin{equation}
\label{xxx1}
g_j(w)=g_j(w'), \ g_j(u)=g_j(u'), \ \ j=1,2\nek i.
\end{equation}
By (\ref{connection between maps}) and (\ref{xxx1}), for every $1\leq j\leq i$ we have 
\[
\begin{split}
g_j(wu)&=h_j\bigl(g_1(w)\nek g_{j-1}(w),g_1(u)\nek g_{j-1}(u)\bigr)g_j(w)g_j(u) \\
&=h_j\bigl(g_1(w')\nek g_{j-1}(w'),g_1(u')\nek g_{j-1}(u')\bigr)g_j(w')g_j(u') \\
&=g_j(w'u').
\end{split}
\]
When $i=t$ this shows that $g(wu)=g(w'u')$, and we are done.

When $i<t$, the assumptions imply that
\[
g_{i+1}(w)\leq_{i+1} g_{i+1}(w'),\   g_{i+1}(u)\leq_{i+1} g_{i+1}(u'),
\]
and at least one of these inequalities is strict.
By Remark \ref{remarks on ordered monoids}(1),  
\[
g_{i+1}(w)g_{i+1}(u)<_{i+1}g_{i+1}(w')g_{i+1}(u').
\]
It follows again from (\ref{connection between maps}), (\ref{xxx1}) and Remark \ref{remarks on ordered monoids}(1) that  $g_{i+1}(wu)<_{i+1}g_{i+1}(w'u')$.
Therefore $g(wu)<_\lex g(w'u')$ in this case.
\end{proof}

\subsection{Free monoids}
Let $X$ be a finite nonempty set and let $X^*$ be the free monoid on $X$.
We consider $X$ as an \textsl{alphabet} and the elements of $X^*$ as (associative) \textsl{words} in $X$.
Thus the product map in $X^*$ is the concatenation $(w,u)\mapsto wu$.
The unit element of $X^*$ is the empty word, denoted $1$.
Note that $X^*$ is cancellative.
We write a word $w$ in $X^*$ as $w=(a_1\cdots a_n)$ with $a_1\nek a_n\in X$ and call $n=|w|$ its \textsl{length}.

Our approach to mildness will be based on constructions of total orders on $X^*$ making it an ordered monoid.
The two most basic constructions are as follows.
Here we fix a total order on $X$.

\begin{exam}
\rm
The lexicographic order on $X^*$ induced by the total order on $X$ makes it an ordered monoid $X^*_\lex$. 
\end{exam}

\begin{exam}
\label{length-lexicographic order}
\rm
Consider $\dbZ_{\geq0}$ as a cancellative ordered monoid with respect to addition and the natural total order.
Take in Lemma \ref{g}
\[
g=(g_1=|\cdot|,g_2=\id)\colon X^*\to\dbZ_{\geq0}\times X^*_\lex, \quad h_1=0, \quad h_2=1.
\]
We obtain that $X^*$ is an ordered monoid with respect to the \textsl{length-lexicographic} order, given by $w\leq_gw'$ if and only if $(|w|,w)\leq(|w'|,w')$ lexicographically.
\end{exam}

\subsection{Lyndon words}
\label{subsection on Lyndon words}
An element $w$ of an ordered monoid $(\calM,\leq)$ is called a \textsl{Lyndon element} if $w\neq1_\calM$ and $w$ is strictly smaller than all its proper right factors.

\begin{examples}
\label{small Lyndon words}
\rm
Let $X$ be a finite nonempty totally ordered alphabet and let $x_1<\cdots<x_d$ be its elements. 
Take $\calM=X^*$ with its lexicographic order.
We then call $w$ as above a \textsl{Lyndon word}.
For instance we have (see \cite{Efrat17}*{Example 2.1}):
\begin{enumerate}
\item[(1)]
All one-letter words $(x_i)$ are Lyndon words.
\item[(2)]
A two-letter word $(x_ix_j)$ is Lyndon if and only if $i<j$.
\item[(3)]
A three-letter word $(x_ix_jx_l)$ is Lyndon if and only if either $i<j,l$ or $i=j<l$.
\end{enumerate}
\end{examples}

\section{Magnus representations}
\label{section on Magnus representations}
In this section $K$ will stand for a profinite commutative unital ring (where ultimately we will take $K=\dbF_p$ with $p$ prime) and $X$ will be a nonempty set.
\subsection{Formal power series}
Let $K\langle X\rangle$ be the free associative $K$-algebra on $X$.
We view its elements as non-commutative polynomials in the variable set $X$ over $K$.
Let $K\langle\langle X\rangle\rangle$ be the $K$-algebra of formal power series in $X$ with coefficients in $K$.
Then  $K\langle X\rangle$ is a subalgebra of $K\langle\langle X\rangle\rangle$ and the free monoid $X^*$ of words in $X$ may be identified with the monomials (with coefficient $1$) in $K\langle X\rangle$.
Thus an element of $K\langle\langle X\rangle\rangle$ can be written as a formal sum $\sum_{w\in X^*}c_ww$ with $c_w\in K$.
It belongs to $K\langle X\rangle$ if and only if at most finitely many of the $c_w$ are nonzero.
The ring $K\langle\langle X\rangle\rangle$ is profinite \cite{Efrat14}*{\S5}.

We write $K\langle\langle X\rangle\rangle^\times$ for the multiplicative group of $K\langle\langle X\rangle\rangle$, and $K\langle\langle X\rangle\rangle^{\times,1}$ for the subset of $K\langle\langle X\rangle\rangle$ consisting of all power series with constant term $1$.
This subset is in fact a subgroup of $K\langle\langle X\rangle\rangle^\times$.
Indeed, it is clearly multiplicatively closed, and if $h\in K\langle\langle X\rangle\rangle$ has constant term $0$, then $1+h$ has a well defined inverse $\sum_{k\geq0}(-1)^kh^k$.
Moreover, $K\langle\langle X\rangle\rangle^{\times,1}$ is a profinite group.

\subsection{Unitriangular matrices}
\label{subsection on unitriangular matrices}
For $n\geq1$ let $\dbU_n(K)$ be the profinite group of all unipotent upper-triangular $(n+1)\times(n+1)$-matrices over $K$.
For every $1\leq i\leq n$, the projection map $\pr_{i,i+1}\colon \dbU_n(K)\to K^+$ on the $(i,i+1)$-entry is a profinite group homomorphism.

The set $Z_n(K)$ of all matrices in $\dbU_n(K)$ which are $1$ on the diagonal, and $0$ at all entries $(i,j)$ with $1\leq i<j\leq n+1$ and $(i,j)\neq(1,n+1)$, is a central subgroup of $\dbU_n(K)$.
It is isomorphic to the additive group $K^+$ of $K$ via
the projection map $\pr_{1,n+1}$ on entry $(1,n+1)$.
Let $\bar\dbU_n(K)=\dbU_n(K)/Z_n(K)$.
Thus there is a group extension
\begin{equation}
\label{group extension gamma}
0\to K^+\to\dbU_n(K)\to\bar\dbU_n(K)\to1.
\end{equation}

Next let $F$ be a profinite group and let $\Lam\colon F\to K\langle\langle X\rangle\rangle^{\times,1}$ be a continuous map.
For $f\in F$ we write 
\[
\Lam(f)=\sum_{w\in X^*}\eps_w(f)w
\]
with $\eps_w(f)\in K$ for every $w\in X^*$ and $\eps_1(f)=1$.
We call the maps $\eps_w\colon F\to K$ the \textsl{coefficient maps} associated with $\Lam$.

Given a word $w=(a_1\cdots a_n)$ of length $n$ in $X^*$, we define a map $\rho_w\colon F\to\dbU_n(K)$ by 
\[
\rho_w(f)=\bigl(\eps_{(a_i\cdots a_{j-1})}(f)\bigr)_{1\leq i\leq j\leq n+1}.
\]

\begin{prop}
\label{Magnus equivalent conditions} 
The following conditions are equivalent:
\begin{enumerate}
\item[(a)]
$\Lam$ is a profinite group homomorphism;
\item[(b)]
For every $f,f'\in F$ and $w\in X^*$ one has 
\[
\eps_w(ff')=\sum_{w=uu'}\eps_u(f)\eps_{u'}(f'),
\]
where the sum is over all the (finitely many) decompositions of $w$ as a concatenation $uu'$ with $u,u'\in X^*$;
\item[(c)]
For every $w\in X^*$, the map $\rho_w$ is a profinite group homomorphism.
\end{enumerate}
\end{prop}
\begin{proof}
See \cite{Efrat24}*{Prop.\ 4.1}.
\end{proof}

In particular, (b) implies that $\eps_{(x)}\colon F\to K^+$ is a profinite group homomorphism for every $x\in X$.

We call a map $\Lam$ satisfying the equivalent conditions of Proposition \ref{Magnus equivalent conditions} a \textsl{generalized Magnus homomorphism} and then call $\rho_w$ the \textsl{generalized Magnus representation} associated with $w$.

\begin{exam}
\label{standard Magnus homomorphism}
\rm
Take $K=\dbF_p$ with $p$ prime, and let $F$ be the free pro-$p$ group on the basis $X=\{x_1\nek x_d\}$.
The continuous homomorphism $\Lam\colon F\to \dbF_p\langle\langle X\rangle\rangle^{\times,1}$ defined on the basis elements by $\Lam(x_i)=1+x_i$ will be called the \textsl{standard pro-$p$ Magnus homomorphism}.
This follows a similar classical construction due to Magnus for discrete free groups; See \cite{Efrat14}*{\S5} for the pro-$p$ case.
As shown by Lazard, $\Lam$ is injective \cite{Serre97}*{I, \S6}.

We note that here the homomorphisms $\eps_{(x_j)}$ are dual to $X$, in the sense that $\eps_{(x_j)}(x_i)$ is $1$ if $i=j$, and is $0$ otherwise.
\end{exam}

\section{Unitriangular cohomology elements}
\label{section on unitriangular cohomology elements}
\subsection{Cohomology and minimal presentations}
Let $p$ be a fixed prime number.
Given closed subgroups $K,K'$ of a pro-$p$ group $G$ we write $K^p$, resp., $[K,K']$, for the closed subgroups generated by all elements $k^p$, resp., $[k,k']$, with $k\in K$ and $k'\in K'$. 
As before, we abbreviate $H^i(G)=H^i(G,\dbF_p)$ and recall that $H^1(G)=\Hom_{\rm cont}(G,\dbF_p)$.

Now assume that $G$ is finitely presented with a fixed minimal presentation 
\begin{equation}
\label{minimal presentation}
1\to R\to F\to G\to 1.
\end{equation}
%   \cite{EfratMinac11}*{Cor.\ 2.2}.
Thus $F$ is a free pro-$p$ group on a nonempty finite basis $X$ and $R$ is the closed normal subgroup of $F$ generated by finitely many elements $r_1\nek r_m$, and the minimality means that $R$ is contained in the Frattini subgroup $F^p[F,F]$.
We denote the coset of $r\in R$ in the $\dbF_p$-linear space $H_2(G,\dbF_p)=R/R^p[F,R]$ by $\bar r$.
We identify the free basis $X$ with the alphabet and the set of non-commuting variables discussed in the previous sections, and use the interplay between these three roles of $X$.

The inflation map $\mathrm{inf}\colon H^1(G)\to H^1(F)$ is an isomorphism, by the minimality of the presentation, and $H^2(F)=0$ \cite{NeukirchSchmidtWingberg}*{Cor.\ 3.9.5}.
It follows from the five term exact sequence in profinite cohomology \cite{NeukirchSchmidtWingberg}*{Prop.\ 1.6.7} that the transgression map $\trg\colon H^1(R)^G\to H^2(G)$ is an isomorphism.
One has a well defined perfect bilinear map
\begin{equation}
\label{transgression pairing}
(\cdot,\cdot)\colon R/R^p[F,R]\times H^2(G)\to\dbF_p, \quad (\bar r,\alp)=(\trg\inv(\alp))(r)
\end{equation}
which is functorial in $R$ in the natural sense \cite{CheboluEfratMinac12}*{\S5}.
We call it the \textsl{transgression pairing}.

\subsection{Unitriangular cohomology elements}
Now let $n\geq1$ be an integer and take in the previous section the ring $K=\dbF_p$.
We abbreviate 
\[
\dbU_n:=\dbU_n(\dbF_p), \quad Z_n:=Z_n(\dbF_p) ,\quad \bar\dbU_n:=\bar\dbU_n(\dbF_p).
\]

We say that a pro-$p$ group homomorphism $\rho\colon F\to \dbU_n$ is \textsl{compatible with the presentation (\ref{minimal presentation})}, if $\rho(R)$ is contained in the subgroup $Z_n(\dbF_p)=\dbZ/p$ of $\dbU_n$.
Then $\rho$ induces a pro-$p$ group homomorphism $\bar\rho\colon G\to\bar\dbU_n$.

\begin{exam}
\label{compatibility for n=2}
\rm
When $n=2$, $\rho$ is always compatible with the presentation.
Indeed, as $R\subseteq F^p[F,F]$ one has
$\rho(R)\subseteq \dbU_2^p[\dbU_2,\dbU_2]=Z_2$.
\end{exam}

Let $\gam_n\in H^2(\bar\dbU_n)$ correspond to the extension (\ref{group extension gamma}) via the Schreier correspondence \cite{NeukirchSchmidtWingberg}*{Th.\ 1.2.4}. 
Every pro-$p$ group homomorphism $\bar\rho\colon G\to\bar\dbU_n$
induces a pullback homomorphism $\bar\rho^*\colon H^2(\bar\dbU_n)\to H^2(G)$.
We call 
\[
\bar\rho^*(\gam_n)\in H^2(G)
\]
the \textsl{unitriangular cohomology element} associated with $\bar\rho$.
This construction is related to the transgression pairing by means of the following key fact:

\begin{prop}
\label{transgression pairing and unitriangular elements}
Consider a pro-$p$ group homomorphism $\rho\colon F\to\dbU_n$ which is compatible with the presentation (\ref{minimal presentation}) and the induced homomorphism $\bar\rho\colon G\to\bar\dbU_n$.
Under the identification $Z_n=\dbF_p$, for every $r\in R$ one has 
\[
\rho(r)=(\bar r,\bar\rho^*(\gam_n)).
\]
\end{prop}
\begin{proof}
Let $R'$ (resp., $\bar\dbU'$),  be the kernel (resp., image) of the induced homomorphism $\rho'\colon F\to\bar\dbU_n$.
We identify $F/R'=\bar\dbU'$.
By the compatibility, $R\leq R'$.
Let $\dbU'$ be the preimage of $\bar\dbU'$ in $\dbU_n$, and let $\gam'\in H^2(\bar\dbU')$ correspond to the extension
\[
0\to\dbZ/p\to\dbU'\to\bar\dbU'\to1.
\]
Then $\gam'=\res(\gam_n)$ under the restriction map $\res\colon H^2(\bar\dbU_n)\to H^2(\bar\dbU')$.

Using the functoriality of the transgression \cite{EfratMinac11}*{\S2B)} we obtain a commutative diagram
\[
\xymatrix{
H^1(R')^F\ar[d]_{\res}\ar[r]^{\trg}&H^2(F/R')\ar[d]^{\inf}&H^2(\bar\dbU')\ar@{=}[l]\ar[d]
  &H^2(\bar\dbU_n)\ar[l]_{\res}\ar[ld]^{\bar\rho^*}\\
H^1(R)^F\ar[r]^{\trg}&H^2(F/R)\ar@{=}[r]&H^2(G).
}
\]
By \cite{Hoechsmann68}*{2.1}, the upper transgression homomorphism maps $\rho|_{R'}$ to $\gam'$.
Hence the lower transgression homomorphism maps $\rho|_R$ to $\bar\rho^*(\gam_n)$.
The assertion now follows from the definition of the transgression pairing.
\end{proof}

\subsection{Massey products}
\label{subsection on Massey products}
Massey products are defined in the general homological context of differential graded algebras.
Here we will be interested only in the profinite cohomology context, and moreover restrict to products into the second cohomology group.
In this case there is an equivalent definition of Massey products, which is due to Dwyer \cite{Dwyer75} in the discrete setting, using unitriangular cohomology elements.
See e.g., \cite{Efrat14} for the profinite setting.

Specifically, take $n\geq2$ and let $\rho_1\nek\rho_n$ be elements of $H^1(G)$.
Let $\bar\rho$ range over all homomorphisms $G\to\bar\dbU_n$ such that $\rho_i=\pr_{i,i+1}\circ\bar\rho$, $i=1,2\nek n$.
The \textsl{$n$-fold Massey product} $\langle\rho_1\nek\rho_n\rangle$ is the (possibly empty) set of all unitriangular cohomology elements $\bar\rho^*(\gam_n)$ for such $\bar\rho$  \cite{Efrat17}*{Example 7.4(2)}.
Thus the $n$-fold Massey product is a multi-valued map
\[
\langle \cdot\nek\cdot\rangle\colon H^1(G)^n\to H^2(G).
\]

When $n=2$ the Massey product is essentially the cup product, namely, $\langle\rho_1,\rho_2\rangle=\{\rho_1\cup\rho_2\}$ (up to a sign convention).
Thus 
\[
\bar\rho^*(\gam_2)=(\pr_{12}\circ\bar\rho)\cup(\pr_{23}\circ\bar\rho),
\]

\subsection{Magnus representations for words}
\label{subsection on Magnus respresentations for words}
Consider a generalized Magnus homomorphism $\Lam\colon F\to \dbF_p\langle\langle X\rangle\rangle^{\times,1}$
with coefficient maps $\eps_w\colon F\to \dbF_p$ for $w\in X^*$.
Given a word $w=(a_1\cdots a_n)\in X^*$ of length $n\geq1$, let 
\[
\rho_w\colon F\to\dbU_n
\]
be the Magnus representation associated with $\Lam$ and $w$ (See \S\ref{subsection on unitriangular matrices}).
When the homomorphism $\rho_w$ is compatible with the presentation (\ref{minimal presentation}) we simply say that \textsl{the word $w$ is compatible with the presentation}.

\begin{lem}
\label{compatible homomorphisms give homomorphic Magnus coef}
Suppose that $w\in X^*$ is compatible with the presentation (\ref{minimal presentation}).
Then:
\begin{enumerate}
\item[(a)]
$\eps_u(r)=0$ for every $r\in R$ and every proper middle factor $u$ of $w$;
\item[(b)]
The restricted map $\eps_w\colon R\to \dbF_p$ is a group homomorphism.
\end{enumerate}
\end{lem}
\begin{proof}
(a) follows immediately from the definition of $\rho_w$.
(b) follows from (a) and from condition (b) of Proposition \ref{Magnus equivalent conditions}.
\end{proof}

Under the compatibility assumption, $\rho_w$ induces as above a pro-$p$ group homomorphism $\bar\rho_w\colon G\to\bar\dbU_n$, which gives rise to a pullback homomorphism 
\[
\bar\rho_w^*\colon H^2(\bar\dbU_n)\to H^2(G).
\]
We define the \textsl{unitriangular cohomology element of $w$} to be 
\[
\alp_w=\bar\rho_w^*(\gam_n)\in H^2(G).
\]
For $1\leq i<n$ the projection of $\rho_w$ on the $(i,i+1)$ entry is $\eps_{(a_i)}$, so
\begin{equation}
\label{Massey elements and Magnus coefficients}
\alp_w\in\langle \eps_{(a_1)}\nek \eps_{(a_n)}\rangle.
\end{equation}
By Proposition \ref{transgression pairing and unitriangular elements}, for every $r\in R$ we have
\begin{equation}
\label{Magnus coefficients as transgression pairings}
\eps_w(r)=(\bar r,\alp_w)\in\dbF_p.
\end{equation}

\section{Coefficient matrices}
\label{section on coefficient matrices}
\subsection{Preliminaries}
Consider a minimal finite presentation $G=F/R$ as in (\ref{minimal presentation}), with $F$ a free pro-$p$ group on the finite basis $X$ with $d$ elements.
Then $d=\dim_{\dbF_p}H^1(G)<\infty$ and $m=\dim_{\dbF_p}H^2(G)<\infty$.
We assume that $m\geq1$, i.e., $G$ is not a free pro-$p$ group \cite{NeukirchSchmidtWingberg}*{Cor.\ 3.9.5}.
We write the $\dbF_p$-linear space $H_2(G,\dbF_p)=R/R^p[F,R]$ additively.
As before, the coset of $r\in R$ in this space will be denoted by $\bar r$, so $\overline{rr'}=\overline r+\overline{r'}$ for $r,r'\in R$.

We recall from \cite{NeukirchSchmidtWingberg}*{Ch.\ III, Cor.\ 3.9.3} the following variant of the ``Frattini argument" for closed normal subgroups of a pro-$p$ group:

\begin{lem}
Elements $r_1\nek r_m$ of $R$ generate it as a closed normal subgroup of $F$ if and only if $\bar r_1\nek\bar r_m$ span $R/R^p[F,R]$.
\end{lem}

\begin{cor}
\label{Rp[F,R]}
The following conditions on $r_1\nek r_m\in R$ are equivalent:
\begin{enumerate}
\item[(a)]
$r_1\nek r_m$ is a minimal system of generators for $R$ as a closed normal subgroup of $F$;
\item[(b)]
The costes $\bar r_1\nek \bar r_m$ form an $\dbF_p$-linear basis of $R/R^p[F,R]$.
\end{enumerate}
\end{cor}

\subsection{Elementary operations}
Fix a generalized Magnus homomorphism $\Lam\colon F\to\dbF_p\langle\langle X\rangle\rangle$ with coefficient maps $\eps_w\colon F\to\dbF_p$, $w\in X^*$.
Let $B$ be a nonempty totally ordered set of nonempty words $w$ in $X^*$ which are compatible with the presentation (\ref{minimal presentation}).
By Lemma \ref{compatible homomorphisms give homomorphic Magnus coef}(b), the restriction $\eps_w\colon R\to\dbF_p$ is a group homomorphism.
Hence for every $r,r'\in R$ and $1\leq k<p$ we have
\begin{equation}
\label{elementary operations for epsilon}
\eps_w(r^k)=k\eps_w(r), \quad \eps_w(r(r')^k)=\eps_w(r)+k\eps_w(r').
\end{equation}

Given a minimal system of generators $\tilde r=(r_1\nek r_m)$ for $R$ as a closed normal subgroup of $F$ we define an $m\times|B|$-matrix over $\dbF_p$ (possibly with infinitely many columns)
\[
M(\tilde r,B)=[\eps_w(r_j)]_{1\leq j\leq m,w\in B}.
\]

On the set of all systems $\tilde r$ as above we have the following three types of \textsl{elementary operations}:
\begin{enumerate}
\item[(i)]
Permutations of $r_1\nek r_m$;
\item[(ii)]
Replacing $r_s$ by $r_s^k$ for some $1\leq s\leq m$ and some $1\leq k<p$;
\item[(iii)]
Replacing $r_s$ by $r_sr_t^k$ for some $s\neq t$ and $1\leq k<p$.
\end{enumerate}
For the induced $m$-tuples $(\bar r_1\nek\bar r_m)$ in $R/R^p[F,R]$, the elementary operations (i)--(iii) result in the operations of permutations of $(\bar r_1\nek\bar r_m)$, multiplication of $\bar r_s$ by the nonzero scalar $k$, and the replacement of $\bar r_s$ by $\bar r_s+k\bar r_t$ for $s\neq t$, respectively. 
 
Similarly, it follows from (\ref{elementary operations for epsilon}) that the elementary operations (i)--(iii) induce the elementary row operations on the matrix $M(\tilde r,B)$ of permutations of the rows, multiplication of row $s$ by the scalar $k$, and adding to the $s$-th row $k$ times the $t$-th row.

\subsection{Unitriangular generators for $H^2(G)$}
\begin{thm}
\label{conditions on the matrix M}
Consider a presentation $G=F/R$ and a basis $X$ for $F$ as above, and let $B$ be a nonempty totally ordered set of nonempty words in $X^*$ which are compatible with the presentation.
The following conditions are equivalent:
\begin{enumerate}
\item[(a)]
There exists a minimal generating system $\tilde r=(r_1\nek r_m)$ for $R$ as a closed normal subgroup of $F$ such that $M(\tilde r,B)$ is in row echelon form with no zero rows.
\item[(b)]
For every minimal generating system $\tilde r=(r_1\nek r_m)$ for $R$ as a closed normal subgroup of $F$ the matrix $M(\tilde r,B)$ has rank $m$.
\item[(c)]
The cohomology elements $\alp_w$, $w\in B$, generate $H^2(G)$.
\end{enumerate}
\end{thm}
\begin{proof}
(a)$\Rightarrow$(b): \quad
Take $\tilde r$ to be a minimal system of generators such that $M(\tilde r,B)$ is in row echelon form with no zero rows, and let $\tilde r'$ be an arbitrary minimal system of generators.
By Corollary \ref{Rp[F,R]}, the induced $m$-tuples $\bar r$ and $\bar r'$ form bases of the $\dbF_p$-linear space $R/R^p[F,R]$.
Therefore there is a finite sequence of elementary operations which takes $\bar r$ to $\bar r'$.
They are induced by elementary operations of types (i)--(iii) which take $\tilde r$ to a system $\tilde r''$ with $\bar r''=\bar r'$.
The latter operations induce elementary row operations taking $M(\tilde r,B)$ to $M(\tilde r'',B)$.
By (\ref{Magnus coefficients as transgression pairings}), we have  $\eps_w(r''_j)=\eps_w(r'_j)$ for every $1\leq j\leq m$ and every  $w\in B$.
Hence $M(\tilde r'',B)=M(\tilde r',B)$, so $M(\tilde r,B)$ and $M(\tilde r',B)$ are row-equivalent, and consequently have the same rank $m$.

\medskip

(b)$\Rightarrow$(c): \quad
Since the transgression pairing is perfect, we need to show that the subgroup of $H^2(G)$ generated by $\alp_w$, $w\in B$, has a trivial annihilator under this pairing.
To this end, take $r\in R$ with $\bar r\neq0$.
We complete $\bar r$ into a linear basis $(\bar r=\bar r_1,\bar r_2\nek\bar r_m)$ of $R/R^p[F,R]$, with $r_2\nek r_m\in R$.
By Corollary \ref{Rp[F,R]}, $\tilde r=(r,r_2\nek r_m)$ is a minimal system of generators for $R$ as a closed normal subgroup of $F$.
By (b), $M(\tilde r,B)$ cannot have a zero row.
Thus there exists $w\in B$ with $\eps_w(r)\neq0$.
By (\ref{Magnus coefficients as transgression pairings}), this means that $(\bar r,\alp_w)\neq0$, so $\bar r$ is not in the annihilator, as desired.

\medskip

(c)$\Rightarrow$(a): \quad
Choose an arbitrary minimal system of generators $\tilde r=(r_1\nek r_m)$ for $R$ as a closed normal subgroup of $F$.
After applying a list of elementary operations of types (i)--(iii), we may assume that $M(\tilde r,B)$ is in row echelon form.

Now for every $1\leq j\leq m$, $\bar r_j\neq0$, by Corollary \ref{Rp[F,R]}.
Since the transgression pairing is perfect, (c) yields $w\in B$ such that 
$(\bar r_j,\alp_w)\neq0$.
By (\ref{Magnus coefficients as transgression pairings}), $\eps_w(r_j)\neq0$.
Thus $M(\tilde r,B)$ has no zero rows.
\end{proof}

\section{Mild groups}
\label{section on mild groups}
\subsection{Strongly free sequences}
Mild pro-$p$ groups were defined by Labute \cite{Labute85}, building on work by Anick \cite{Anick82} in the discrete case.
Labute's original definition was in the context of gradings obtained from the lower $p$-central filtration, and in a Lie-theoretic setting.
Later works used mostly gradings obtained from the $p$-Zassenhaus filtrations (see \S\ref{section on the Zassenhaus filtration} below).
In this section we give a uniform treatment of both situations using the notion of an initial form map.

For the definition of mildness, we first recall the \textsl{Golod--Shafarevich inequality} for finitely presented graded algebras \cite{Ershov12}:

Let again $X$ be a finite nonempty set of $d$ elements.
Let $\tau\colon X\to\dbZ_{>0}$ be a map with induced monoid morphism $\tau^*\colon X^*\to\dbZ_{\geq0}$, and let $\dbF$ be a field.
We make the free associative algebra $\dbF\langle X\rangle$ a graded algebra $\dbF\langle X\rangle^\tau$ where the degree $n$ component is the $\dbF$-module generated by all words $w\in X^*$ with $\tau^*(w)=n$.
We call its nonzero elements \textsl{$\tau$-homogeneous} of degree $n$.
We further write $\deg^\tau\rho$ for the degree of a $\tau$-homogeneous element $\rho$.

Given $\tau$-homogenous elements $\rho_1\nek\rho_m$ of positive degree in $\dbF\langle X\rangle^\tau$, let $\calR$ be the two-sided (graded) ideal they generate.
It gives rise to a graded $\dbF$-algebra $\dbF\langle X\rangle^\tau/\calR$.
Let $a_n$ be the dimension of the degree $n$ component of $\dbF\langle X\rangle^\tau/\calR$.
The \textsl{Hilbert series} of $\dbF\langle X\rangle^\tau/\calR$ is the formal power series
\[
H(z)=\sum_{n\geq0}a_nz^n.
\]
Then the following \textsl{Golod--Shafarevich inequality} holds coefficient-wise (See \cite{Ershov12}*{Th.\ 2.1}, \cite{Gartner15}*{Lemma 2.5}):
\[
H(z)\Bigl(1-\sum_{x\in X}z^{\tau(x)}+\sum_{j=1}^mz^{\deg^\tau\rho_j}\Bigr)\geq1.
\]
When this is an equality, $\rho_1\nek\rho_m$ are said to be \textsl{strongly free} with respect to $\tau$.

Equivalently, $\rho_1\nek\rho_m$ are strongly free if and only if $\calR/\calR I$ is a free $\dbF\langle X\rangle^\tau/\calR$-module on the images of $\rho_1\nek\rho_m$, where $I=\langle X\rangle$ is the augmentation ideal of $\dbF\langle X\rangle$ (\cite{Forre11}*{\S1}, \cite{Gartner15}*{Prop.\ 2.9}).
An alternative Lie-algebraic definition is given in \cite{Labute06} (see also \cite{Forre11}*{Remark 1.6}).  
As explained in \cite{Anick82}, the notion of a strongly free sequence may be viewed as the non-commutative analog of the notion of a regular sequence in commutative rings.

\subsection{Combinatorially free sequences}
Following Anick \cite{Anick82}, we call a sequence $w_1\nek w_m$ of nonempty words in $X^*$ \textsl{combinatorially free} if:
\begin{enumerate}
\item
No $w_i$ is a middle factor of $w_j$ for every $i\neq j$; and 
\item
No proper left factor of $w_i$ is a proper right factor of $w_j$ for every (possibly equal) $i,j$.
\end{enumerate}
Trivially, a subsequence of a combinatorially free sequence is also combinatorially free.

Given a total order $\leq$ on $X^*$, the \textsl{leading term} of a nonzero $\tau$-homogenous element $\gam=\sum_wc_ww$ in $\dbF\langle X\rangle^\tau$, with $0\neq c_w\in\dbF$, is the maximal $w$ in this sum with respect to $\leq$.

The following result was proved by Anick \cite{Anick82}*{Th.\ 3.1 and Th.\ 3.2} for the length-lexicographic order (Example \ref{length-lexicographic order}).
Forr\'e \cite{Forre11}*{Th.\ 2.6} observed that it holds for arbitrary total orders $\leq$ such that $(X^*,\leq)$ is an ordered monoid.

\begin{prop}
\label{Anick criterion}
Let $\leq$ be a total order on $X^*$ such that $(X^*,\leq)$ is an ordered monoid, and let $\tau\colon X\to\dbZ_{\geq0}$ be a map.
Suppose that $\gam_1\nek\gam_m$ are $\tau$-homogeneous elements of positive $\tau$-degree in $\dbF\langle X\rangle^\tau$. 
Let $w_1\nek w_m$ be the leading terms of $\gam_1\nek \gam_m$, respectively, with respect to $\leq$.
If the sequence $w_1\nek w_m$ is combinatorially free, then the sequence $\gam_1\nek\gam_m$ is strongly free.
\end{prop}

\subsection{Initial forms}
\label{subsection on initial forms}
Now take $\dbF$ to be a field of characteristic $p$ and consider again a minimal presentation $G=F/R$ as in (\ref{minimal presentation}), where $F$ is a free pro-$p$ group on a finite basis $X$ of size $d$, and $R$ is its closed normal subgroup generated by relations $r_1\nek r_m$. 
Let $\tau\colon X\to\dbZ_{>0}$ be a map as before.
Consider a map 
\[
I\colon R\to \dbF\langle X\rangle,
\]
which we call an \textsl{initial form map}. 

\begin{defin}
\label{Definition of mildness}
\rm
We say that the above presentation of $G$ is \textsl{mild with respect to $\tau$ and $I$} if there are generators $r_1\nek r_m$ as above such that $I(r_1)\nek I(r_m)$ are $\tau$-homogenous elements of positive $\tau$-degree in $\dbF\langle X\rangle^\tau$ which form a strongly free sequence with respect to $\tau$.
\end{defin}

Examples of initial form maps related to natural filtrations of $F$ will be given in \S\ref{section on the Zassenhaus filtration}.

\section{The mildness criterion}
\label{section on the mildness criterion}
We fix a finite nonempty set $X$, let $\tau\colon X\to\dbZ_{>0}$ be a map, and let $\tau^*\colon X^*\to\dbZ_{\geq0}$ the induced monoid morphism.
In particular, for the constant map $\tau=1$, the induced map $\tau^*$ is the length map.

Let $G$ be a finitely presented pro-$p$ group with minimal presentation $G=F/R$ as in (\ref{minimal presentation}), where $F$ is the free pro-$p$ group on the basis $X$. 
We further fix a generalized Magnus homomorphism $\Lam\colon F\to\dbF_p\langle\langle X\rangle\rangle^{\times,1}$.

We can now state our main criterion for mildness (Theorem A):

\begin{thm}
\label{main theorem on mildness}
Let $\leq$ be a total order on $X^*$ such that $(X^*,\leq)$ is an ordered monoid.
Let $A$ be a finite set of nonempty words in $X^*$ which are compatible with the presentation $G=F/R$ and which have the same $\tau$-degree. 
Suppose that $B$ is a subset $A$ such that:
\begin{enumerate}
\item[(a)]
The set $B$ is combinatorially free.
\item[(b)]
If $w_1\in B$ and $w_2\in A$ satisfy $\alp_{w_1},\alp_{w_2}\neq0$ and $w_1\leq w_2$, then also $w_2\in B$.
\item[(c)]
The cohomology elements $\alp_w$, $w\in B$, generate $H^2(G)$.
\end{enumerate}
Consider the initial form map 
\[
I\colon R\to\dbF_p\langle X\rangle^\tau, \quad
I(r)=\sum_{w\in A}\eps_w(r)w.
\]
Then there are generators $r_1\nek r_m$ of $R$ as a closed normal subgroup of $F$ such that the sequence $I(r_1)\nek I(r_m)$ is strongly free in $\dbF_p\langle X\rangle^\tau$.
In particular, the presentation $G=F/R$ is mild with respect to $\tau$ and $I$. 
\end{thm}
\begin{proof}
We may apply Theorem \ref{conditions on the matrix M} with respect to $B$ totally ordered by the \textsl{opposite} order $\leq^{\rm op}$ of $\leq$.
In view of (c), it yields a minimal system $\tilde r=(r_1\nek r_m)$ of generators for $R$ as a closed normal subgroup of $F$ such that the matrix $M(\tilde r,B)=[\eps_w(r_j)]_{1\leq j\leq m,w\in B}$ is in a row echelon form with no zero rows.

For every $1\leq j\leq m$ let 
\[
A_j=\{w\in A\ |\ \eps_w(r_j)\neq0\}.
\]
By (\ref{Magnus coefficients as transgression pairings}) and the compatibility, for every $w\in A_j$ we have $\eps_w(r_j)=(\bar r_j,\alp_w)$, so $\alp_w\neq0$. 
Since $B\subseteq A$ and $M(\tilde r,B)$ has no zero rows, $A_j\cap B\neq\emptyset$. 
It follows from (b) that $w_j:=\max(A_j)\in B$ (the maximum taken with respect to $\leq$).
Thus $w_j$ corresponds to the leading term in the $j$th row of $M(\tilde r,B)$.
Since $M(\tilde r,B)$ is in a row echelon form, $w_1\nek w_m$ are distinct.
By (a), they are combinatorially free.

Now consider
\[
I(r_j)=\sum_{w\in A}\eps_w(r_j)w=\sum_{w\in A_j}\eps_w(r_j)w
\]
as homogenous elements of positive degree in $\dbF_p\langle X\rangle^\tau$.
Then the leading term of $I(r_j)$ with respect to $\leq$ is $w_j$.
Anick's criterion (Proposition \ref{Anick criterion}) therefore implies that the sequence $I(r_1)\nek I(r_m)$ is strongly free.
Consequently, the presentation $G=F/R$ is mild with respect to $\tau$ and $I$. 
\end{proof}

\section{Partitions and mildness}
\label{section on partictions and mildness}
In this section we use partitions of the alphabet $X$ to produce ordered structures on the monoid $X^*$ with special properties.
In \S\ref{section on Massey product criteria} we will apply these constructions to obtain mildness criteria based on Massey products.

\subsection{An auxiliary construction} 
\label{subsection on an auxiliary construction}
We fix a total order $\leq$ on $X$ and a map $\tau\colon X\to\dbZ_{\geq0}$.
Given an additional map $\sig\colon X\to\dbZ_{\geq0}$, we set as before $\sig^*(w)=\sum_{i=1}^n\sig(a_i)$ for a word $w=(a_1\cdots a_n)$.
We also define a map $\sig^\#\colon X^*\to\dbZ_{\ge0}$ by setting
\[
\sig^\#(w)=\sum_{i=1}^n\sig(a_i)(\tau(a_1)+\cdots+\tau(a_i)).
\]
In particular, for the empty word $1$ one has $\sig^*(1)=\sig^\#(1)=0$.
We observe that for every $w,u$ one has
\begin{equation}
\label{almost multiplicativity}
\begin{split}
\sig^*(wu)&=\sig^*(w)+\sig^*(u),\\
\sig^\#(wu)&=\tau^*(w)\sig^*(u)+\sig^\#(w)+\sig^\#(u).
\end{split}
\end{equation}

Given a sequence of maps $\sig_1\nek\sig_s\colon X\to\dbZ_{\geq0}$, we define a map $g\colon X^*\to \dbZ_{\geq0}^{1+2s}\times X^*_{\rm lex}$ by
\begin{equation}
\label{g in generalized Gartners construction}
g=(\tau^*,\sig_s^*,\sig_s^\#,\sig_{s-1}^*,\sig_{(s-1)}^\#\nek \sig_1^*,\sig_1^\#,\id_{X^*}).
\end{equation}
As in \S\ref{subsection on a lexicographic construction}, we order $X^*$ by
\[
w\leq_gw'\ \Leftrightarrow\ g(w)\leq_{\rm lex}g(w'),
\]
where $\leq_{\rm lex}$ is the lexicographic order on $\dbZ_{\geq0}^{1+2s}\times X^*_{\rm lex}$.

\begin{lem}
\label{g gives an ordered monoid}
$(X^*,\leq_g)$ is an ordered monoid.  
\end{lem}
\begin{proof}
Since the rightmost component of $G$ is the identity, $g$ is injective.
Now apply Lemma \ref{g} using (\ref{almost multiplicativity}).  
\end{proof}

The case $s=1$ of the lemma recovers \cite{Gartner15}*{Lemma 3.7}.

\subsection{The case $\tau=1$}
We now focus on the constant weight map $\tau=1$, i.e., $\tau(x)=1$ for every $x\in X$.
Recall that $1^*(w)=|w|$ for every word $w$.

\begin{prop}
\label{A max}
Let $k_0,k_1\nek k_s$ be nonnegative integers and set $n=k_0+k_1+\cdots+k_t$.
Suppose that $X=Y_0\discup Y_1\discup\cdots\discup Y_s$, and write $\sig_j\colon X\to\{0,1\}$ for the characteristic map of $Y_j$, $j=1,2\nek s$.
Let $C\subseteq X^n$, let $A$ be the set of all words $w=(a_1\cdots a_n)\in C$ such that 
\[
\bigl|\{1\leq i\leq n\ |\ a_i\in Y_j\}\bigr|\leq k_j, \quad j=1,2\nek s,
\]
and let 
\[
B=C\cap \prod_{j=0}^sY_j^{k_j}.
\] 
Finally, let $g$ be as in (\ref{g in generalized Gartners construction}).
Then for every $w_1\in B$ and $w_2\in A$ with $w_1\leq_gw_2$ we have $w_2\in B$.
\end{prop}
\begin{proof}
First we note that $B\subseteq A$.
For $0\leq r\leq s$ set
\[
Z_r=(Y_0\cup Y_1\cup\cdots\cup Y_r)^{k_0+\cdots+k_r}\times\prod_{j=r+1}^sY_j^{k_j}.
\]
One has 
\[
B\subseteq\prod_{j=0}^sY_j^{k_j}=Z_0\subseteq Z_1\subseteq\cdots \subseteq Z_{s-1}\subseteq Z_s=X^n.
\]

We claim that the maps
\begin{equation}
\label{list of maps}
|\cdot|,\sig_s^*, \sig_{s-1}^*\nek  \sig_{r+1}^*, \sig_s^\#, \sig_{s-1}^\# \nek \sig_{r+1}^\#
\end{equation}
are constant on $Z_r$.
Indeed, for $w\in Z_r$ and $r+1\leq j\leq s$ one has
\[
|w|=n, \quad \sig_j^*(w)=k_j,
\]
\[
\sig_j^\#(w)=\sum\{i\ |\ k_0+\cdots+k_{j-1}<i\leq k_0+\cdots+k_j\}.
\]
 
Now take as in the assertion 
\[
w_1=(a_{11}\cdots a_{1n})\in B, \quad w_2=(a_{21}\cdots a_{2n})\in A
\]
with $w_1\leq_gw_2$.
We also take $0\leq r\leq s$ minimal such that $w_2\in Z_r$.
Since $w_1\in B\subseteq Z_r$, and by the claim, the maps in (\ref{list of maps}) coincide on $w_1$ and $w_2$. 

Suppose that $r\geq1$.
Then $\sig_r^*(w_1)=k_r$ and $\sig_r^*(w_2)\leq k_r$, by the definition of $A$.
If $\sig_r^*(w_2)<k_r$, then $w_2<_gw_1$, a contradiction.
Therefore $\sig_r^*(w_2)=k_r$, i.e., $w_2$ has exactly $k_r$ letters from $Y_r$.

Moreover, as $w_2\in Z_r$, these letters can appear only among the first $k_0+\cdots+k_r$ places. 
By the minimality of $r$, $w_2\not\in Z_{r-1}$, so these letters are not located at the $k_r$ rightmost of these, that is, at places
\[
k_0+\cdots+k_{r-1}+1\nek k_0+\cdots+k_r
\]
(as is the case for $w_1$).
Hence 
\[
\sig_r^\#(w_2)=\sum_{a_{2i}\in Y_r}i<\sum_{a_{i1}\in Y_r}i=\sig_r^\#(w_1), 
\] 
and we obtain again the contradiction $w_2<_gw_1$.

Consequently, $r=0$, so $w_2\in Z_0=\prod_{j=0}^sY_j^{k_j}$.
As $w_2\in A\subseteq C$, we deduce that $w_2\in B$.
\end{proof}

\subsection{A mildness criterion}
Let $G=F/R$ be again a minimal finite presentation as in (\ref{minimal presentation}), where $F$ is the free pro-$p$ group on the finite nonempty basis $X$.
We deduce our mildness criterion for this setting:

\begin{thm}
\label{partition criterion}
Let $s,k_0,k_1\nek k_s$ be positive integers and set $n=k_0+k_1+\cdots+k_s$.
Suppose that $X=Y_0\discup Y_1\discup\cdots\discup Y_s$ where $Y_0,Y_1\nek Y_s$ are nonempty and finite.
Let $C$ be the set of all words in $X^n$ which are compatible with the presentation $G=F/R$.
Assume that:
\begin{enumerate}
\item[(i)]
For every word $w\in C$ which contains more than $k_j$ letters from $Y_j$ for some $1\leq j\leq s$, one has $\alp_w=0$.
\item[(ii)]
The cohomology elements $\alp_w$, for $w\in C\cap \prod_{j=0}^sY_j^{k_j}$, generate $H^2(G)$.
\end{enumerate}
Then $G$ is mild with respect to the initial form map 
\[
I\colon R\to\dbF_p\langle X\rangle, \quad 
I(r)=\sum_{w\in C}\eps_w(r)w.
\]
\end{thm}
\begin{proof}
We fix an arbitrary total order on $X$ and define a total order $\leq_g$ on $X^*$ as in \S\ref{subsection on an auxiliary construction}.
By Lemma \ref{g gives an ordered monoid}, $(X^*,\leq_g)$ is an ordered monoid.

Let $A$ be the set of all words $w=(a_1\cdots a_n)\in C$ such that 
\[
|\{1\leq i\leq n\ |\ a_i\in Y_j\}|\leq k_j, \quad j=1,2\nek s.
\]
By (i),
\begin{equation}
\label{A C equal for nonzero alpha}
\{w\in A\ |\ \alp_w\neq0\}=\{w\in C\ |\ \alp_w\neq0\}.
\end{equation}
Let $B=C\cap \prod_{j=0}^sY_j^{k_j}$, so $B\subseteq A$.
We verify conditions (a)--(c) of Theorem \ref{main theorem on mildness} with $\tau=1$.

Since all words in $B$ have length $n$, non of them can be a proper middle factor of another word in $B$.
Since the $Y_j$ are nonempty and disjoint, and $s,k_1\nek k_s\geq1$, no proper left factor of a word in $B$ is a proper right factor of a word in $B$.
Therefore $B$ is combinatorically free, so condition (a) is satisfied.

Condition (b) is just Proposition \ref{A max}.

Finally, condition (c) is just assumption (ii).

Theorem \ref{main theorem on mildness} therefore implies that $G$ is mild with respect to the initial form map \[
I\colon R\to\dbF_p\langle X\rangle, \quad 
I(r)=\sum_{w\in A}\eps_w(r)w.
\]
Moreover, using (\ref{Magnus coefficients as transgression pairings}) and (\ref{A C equal for nonzero alpha}) we obtain:
\[
I(r)=\sum_{w\in A}\eps_w(r)w
=\sum_{w\in A,\alp_w\neq0}\eps_w(r)w
=\sum_{w\in C,\alp_w\neq0}\eps_w(r)w
=\sum_{w\in C}\eps_w(r)w.
\qedhere
\]
\end{proof}

\section{The Zassenhaus filtration}
\label{section on the Zassenhaus filtration}
Initial form maps (see \S\ref{subsection on initial forms}) provide a bridge between the free pro-$p$ group $F$ and the graded algebra $\dbF_p\langle X\rangle$.
They can be conveniently constructed using natural filtrations on $F$, such as the $p$-Zassenhaus filtration.
In this section we recall some basic facts about this and other standard filtrations of pro-$p$ groups.
The resulting initial form maps will be used in the next section to provide concrete applications of Theorem \ref{partition criterion} using Massey products.

\subsection{Definition and main properties}
The \textsl{$p$-Zassenhaus filtration}  $G_{(n)}$, $n=1,2\nek$ of a pro-$p$ group $G$ (also called the \textsl{modular dimension filtration} \cite{DixonDuSautoyMannSegal99}*{Ch.\ 11})  is defined inductively by
\[
G_{(1)}=G, \quad G_{(n)}=(G_{(\lceil n/p\rceil)})^p\prod_{n=i+j}[G_{(i)},G_{(j)}]
\]
for $n\geq2$.
One has $G_{(n)}\supseteq G_{(n+1)}$.

In particular, take a free pro-$p$ group $F$ on the basis $X$.
Let  $\Lam\colon F\to\dbF_p\langle\langle X\rangle\rangle^{\times,1}$ be the standard Magnus homomorphism, and $\eps_w\colon F\to\dbF_p$, $w\in X^*$, its coefficient maps.
Then $F_{(n)}$ consists of all $f\in F$ such that $\eps_w(f)=0$ for all $w\in X^*$ with $1\leq|w|<n$ (see \cite{Morishita12}*{\S\ 8.3}, \cite{Vogel05}*{Th.\ 2.19(ii)}, \cite{Efrat14}*{Prop.\ 6.2}).
By the injectivity of $\Lam$ (see \S\ref{section on unitriangular cohomology elements}), $\bigcap_{n\geq1}F_{(n)}=\{1\}$.

\subsection{The $p$-Zassenhaus invariant}
\label{subsection on the p-Zassenhaus invariant}
Let $G$ be a finitely presented pro-$p$ group.
The \textsl{Zassenhaus invariant} $\grz(G)$ of $G$ is the supremum of all $n\in\dbZ_{\geq0}\cup\{\infty\}$ such that there is a minimal presentation $G=F/R$ with $R\subseteq F_{(n)}$ (with $F_{(\infty)}$ understood as $\{1\}$).
By the minimality, $R\subseteq F_{(2)}$, so $\grz(G)\geq2$.

By \cite{Vogel04}*{Cor.\ 1.2.9}, $\grz(G)$ is the supremum of all $n$ such that the $k$-fold Massey product $\langle\cdot\nek\cdot\rangle\colon H^1(G)^k\to H^2(G)$ (see \S\ref{subsection on Massey products}) is identically zero for every $2\leq k<n$.
Further, when $\grz(G)=n$, the $n$-fold Massey product is uniquely defined and is multi-linear \cite{Vogel04}*{Lemma 1.2.5}.
We therefore write its image as an element, rather than a subset, of $H^2(G)$.

\begin{lem}
\label{Zassenhaus invariant and compatibility}
Let $G=F/R$ be a minimal presentation, with $F$ free pro-$p$ on the basis $X$, and suppose that $\grz(G)=n$.
Then:
\begin{enumerate}
\item[(a)]
For every $r\in R$ and every word $w\in X^*$ of length $1\leq|w|<n$ one has $\eps_w(r)=0$.
\item[(b)]
Every word $w$ of length $n$ is compatible with this representation.
\end{enumerate}
\end{lem}
\begin{proof}
(a) follows from the description of $F_{(n)}$ in terms of the (standard) Magnus coefficients.
(b) follows from (a) and the definition of $\rho_w$.
\end{proof}
 
\subsection{The Zassenhaus initial form map}
For the notion of \textsl{$p$-restricted algebras} we refer to \cite{DixonDuSautoyMannSegal99}*{\S12.1}, \cite{Jacobson62}*{Ch.\ 5, \S7}, or \cite{Gartner11}*{\S1.2}.

Given a pro-$p$ group $G$, the direct sum $\gr G=\bigoplus_{n\geq0}G_{(n)}/G_{(n+1)}$ is a graded $p$-restricted Lie $\dbF_p$-algebra, with Lie bracket induced by the commutator map, and structural map $(\cdot)^{[p]}$ induced by the map $g\mapsto g^p$ on $G$ \cite{DixonDuSautoyMannSegal99}*{Th.\ 12.8}.
Moreover, when $F$ is a free pro-$p$ group on $X$, $\gr F$ is a free graded $p$-restricted Lie $\dbF_p$-algebra, with basis consisting of the cosets in $F/F_{(2)}$ of all $x\in X$ \cite{Gartner11}*{Th.\ 1.3.8}.

We filter the algebra $\dbF_p\langle\langle X\rangle\rangle$ by setting $\dbF_p\langle\langle X\rangle\rangle_{(n)}$ to be the additive group of all formal power series $\sum_{|w|\geq n}c_ww$ with $c_w\in \dbF_p$.
This gives rise to a graded $\dbF_p$-algebra 
\[
\gr(\dbF_p\langle\langle X\rangle\rangle)=\bigoplus_{n\geq0}\gr_n(\dbF_p\langle\langle X\rangle\rangle), \quad
\gr_n(\dbF_p\langle\langle X\rangle\rangle=\dbF_p\langle\langle X\rangle\rangle_{(n)}/\dbF_p\langle\langle X\rangle\rangle_{(n+1)}.
\]
We further grade $\dbF_p\langle X\rangle$ by $\deg x=1$ for all $x\in X$.
Then there is a graded ring isomorphism 
\[
\gr(\dbF_p\langle\langle X\rangle\rangle)\xrightarrow{\sim} \dbF_p\langle X\rangle, \quad \bar x\mapsto x,
\]
where $\bar x$ is the coset of $x\in X$ in $\gr_1(\dbF_p\langle\langle X\rangle\rangle)$.
Considering $\dbF_p\langle X\rangle$ as a $p$-restricted Lie algebra in the natural manner, $\Lam$ induces a monomorphism of graded $p$-restricted Lie $\dbF_p$-algebras  
\[
\gr\,\Lam\colon\gr F\to\gr(\dbF_p\langle\langle X\rangle\rangle)\xrightarrow{\sim}\dbF_p\langle X\rangle, \quad
xF_{(2)}\mapsto x.
\]
Moreover, $\dbF_p\langle X\rangle$ is the $p$-restricted enveloping algebra of $\gr\, F$ under this embedding \cite{Gartner11}*{Prop.\ 1.2.7}.

Now for every $1\neq f\in F$ take $n\geq1$ maximal with $f\in F_{(n)}$.
We call the coset $fF_{(n+1)}$ of $f$ in $F_{(n)}/F_{(n+1)}$ the \textsl{initial form} of $f$, considered as an element of $\dbF_p\langle X\rangle$ via $\gr\,\Lam$.
For a minimal presentation $G=F/R$ as above, and taking $f=r\in R$, we define the \textsl{Zassenhaus initial form map}
by
\[
I\colon R\to\dbF_p\langle X\rangle, \quad r\mapsto (\gr\,\Lam)(rF_{(n+1)}),
\]
where we also set $I(1)=0$. 
Thus $I(r)=\sum_{|w|=n}\eps_w(r)w$.

\begin{rem}
\rm
In \cite{Gartner11} some of these connections are generalized to the context of maps $\tau\colon X\to\dbZ_{\geq1}$ and  the filtration $F_{(\tau,n)}$, $n=1,2\nek$ where $F_{(\tau,n)}$ consists of all $f\in F$ such that $\eps_w(f)=0$ whenever $w\in X^*$ and $1\leq \tau^*(w)<n$.
\end{rem}

\begin{rem}
\rm
In a similar way, we can construct an initial form map using the \textsl{lower $p$-central filtration} $G^{(n)}$, $n=1,2\nek$ of the pro-$p$ group $G$, defined inductively by
\[
G^{(1)}=G, \quad  G^{(n+1)}=(G^{(n)})^p[G,G^{(n)}]
\] 
for $n\geq1$.
For this, assume first that $p\neq2$.
The graded $\dbF_p$-module $\bigoplus_{n\geq1} G^{(n)}/G^{(n+1)}$ has the structure of a graded Lie $\dbF[\pi]$-algebra, where the Lie bracket is induced by the commutator map, and $\pi$ is an indeterminate which acts on the $n$th component by $\pi(gG^{(n+1)})=g^pG^{(n+1)}$.  

When $G=F$ is a free pro-$p$ group on the basis $X$, this algebra is the free graded Lie $\dbF_p[\pi]$-algebra on the images of the basis elements $x\in X$ in $F/F^{(2)}$ \cite{NeukirchSchmidtWingberg}*{Ch.\ III, p.\ 222, Remark}. 
It embeds in its enveloping algebra, which is $(\dbF_p[\pi])\langle X\rangle$ \cite{SerreLie}*{I, Th.\ 4.2}.
For a presentation $G=F/R$ as before, this gives rise in a similar way to the \textsl{lower $p$-central initial form map} $I\colon R\to (\dbF_p(\pi))\langle X\rangle$.

The remaining case where $p=2$ can be treated using Lazard's notion of \textsl{mixed Lie algebras} \cite{Lazard65}*{Ch.\ 2, \S3} -- see \cite{LabuteMinac11}.
\end{rem}

\subsection{Shuffle relations}
Let again $X=\{x_1\nek x_d\}$.
The structure of $H^2(F/F_{(n)})$, for $F$ free pro-$p$ on basis $X$, was computed in \cite{Efrat23} in terms of the shuffle algebra on $X$.
We now recall a part of this structure, that will be needed in the next section.

We write $\bigoplus_{w\in X^n}\dbZ w$ for the free $\dbZ$-module on the words in $X^*$ of length $n$. 
The \textsl{shuffle product} $u\sha v$ of words $u=(a_1\cdots a_s)$, $v=(a_{s+1}\cdots a_{s+t})$ in $X^*$  is defined by
\[
u\sha v=\sum_\pi(a_{\pi^{-1}(1)}\cdots a_{\pi^{-1}(s+t)})\in\bigoplus_{w\in X^{s+t}}\dbZ w,
\]
where $\pi$ ranges over all permutations of $1,2,\ldots, s+t$ such that 
\[
\pi(1)<\cdots<\pi(s), \quad \pi(s+1)<\cdots<\pi(s+t).
\]
Let $\Shu_n(X)$ be the $\dbZ$-submodule of $\bigoplus_{w\in X^n}\dbZ w$ generated by all shuffle products $u\sha v$, where  $n=|u|+|v|$, $|u|,|v|\geq1$.

As before, let $F$ be the free pro-$p$ group on the basis $X$, and let $\chi_1\nek\chi_d$ be an $\dbF_p$-linear basis of $H^1(F)$ which is dual to $x_1\nek x_d$.
For $n\geq2$ one has $H^1(F/F_{(n)})\isom H^1(F)$ via inflation, so we may consider $\chi_1\nek\chi_d$ also as a basis of $H^1(F/F_{(n)})$.
Further, $\grz(F/F_{(n)})=n$, so the $n$-fold Massey product for $F/F_{(n)}$ is a single-valued multi-linear map (See \S\ref{subsection on the p-Zassenhaus invariant}).
Therefore for every word $w=(x_{l_1}\nek x_{l_n})$ in $X^n$ we may view $\bar\alp_w=\langle\chi_{l_1}\nek\chi_{l_n}\rangle$ as an element of $H^2(F/F_{(n)})$.

By \cite{Efrat23} or \cite{Vogel04}*{Example 1.2.11}, the $\dbZ$-module homomorphism
\[
\bigoplus_{w\in X^n}\dbZ w\to H^2(F/F_{(n)}), \quad(k_w)_{w\in X^n}\mapsto \sum_{w\in X^n}k_w\bar\alp_w. 
\] 
is trivial on $\Shu_n(X)$.
When $n<p$ there is an induced monomorphism
\begin{equation}
\label{monomorphism words to cohomology}
\Bigl(\bigl(\bigoplus_{w\in X^n}\dbZ w\bigr)/\Shu_n(X)\bigr)\tensor\dbF_p\to H^2(F/F_{(n)}).
\end{equation}

\begin{rem}
\label{image of Lyndon words in cohomology}
\rm
For a fixed total order on $X$, the cosets of the Lyndon words $w$ in $X^n$ form an $\dbF_p$-linear basis for the left linear space in (\ref{monomorphism words to cohomology}) \cite{Efrat23}.
Therefore the image of the monomorphism (\ref{monomorphism words to cohomology}) is generated by all $n$-fold Massey products $\bar\alp_w$ corresponding to Lyndon words $w\in X^*$ of length $n$.
\end{rem}

In particular, we obtain \textsl{shuffle relations} among $n$-fold Massey products in $H^2(F/F_{(n)})$, e.g., as follows:

\begin{examples}
\label{shuffle relations examples}
\rm
\begin{enumerate}
\item[(1)]
For $n=3$, $u=(x_{l_1}x_{l_2})$ and $v=(x_{l_3})$ one has
\[
\langle\chi_{l_1},\chi_{l_2},\chi_{l_3}\rangle+\langle\chi_{l_1},\chi_{l_3},\chi_{l_2}\rangle
+\langle\chi_{l_3},\chi_{l_1},\chi_{l_2}\rangle=0.
\]
\item[(2)] 
For $n=4$, $u=(x_{l_1})$ and $v=(x_{l_2}x_{l_3}x_{l_4})$ one has
\[
\langle\chi_{l_1},\chi_{l_2},\chi_{l_3},\chi_{l_4}\rangle+
\langle\chi_{l_2},\chi_{l_1},\chi_{l_3},\chi_{l_4}\rangle+
\langle\chi_{l_2},\chi_{l_3},\chi_{l_1},\chi_{l_4}\rangle+
\langle\chi_{l_2},\chi_{l_3},\chi_{l_4},\chi_{l_1}\rangle=0.
\]
\item[(3)]
For $n=4$, $u=(x_{l_1}x_{l_2})$ and $v=(x_{l_3}x_{l_4})$ one has
\[
\begin{split}
&\langle\chi_{l_1},\chi_{l_2},\chi_{l_3},\chi_{l_4}\rangle+
\langle\chi_{l_1},\chi_{l_3},\chi_{l_2},\chi_{l_4}\rangle+
\langle\chi_{l_1},\chi_{l_3},\chi_{l_4},\chi_{l_2}\rangle\\
&+\langle\chi_{l_3},\chi_{l_1},\chi_{l_2},\chi_{l_4}\rangle+
\langle\chi_{l_3},\chi_{l_1},\chi_{l_4},\chi_{l_2}\rangle+
\langle\chi_{l_3},\chi_{l_4},\chi_{l_1},\chi_{l_2}\rangle=0.
\end{split}
\]
\item[(4)] 
For every $n\geq2$ the shuffle relations imply in a formal way that
\[
\langle \chi_{l_1}\nek\chi_{l_n}\rangle=(-1)^{n-1}\langle\chi_{l_n}\nek\chi_{l_1}\rangle.
\]
See \cite{Efrat17}*{Example 9.7}.
\end{enumerate}
\end{examples}

Now given a pro-$p$ group $G$ with $\grz(G)=n$ and a minimal finite presentation $G=F/R$ with $R\subseteq F_{(n)}$, one has $F/F_{(n)}\isom G/G_{(n)}$.
The composed homomorphism
\[
H^2(F/F_{(n)})\xrightarrow{\sim} H^2(G/G_{(n)})\xrightarrow{\inf}H^2(G)
\]
maps $\bar\alp_w$, for a word $w=(x_{l_1}\cdots x_{l_n})$, to $\langle\chi_{l_1}\nek\chi_{l_n}\rangle\in H^2(G)$, where here $\chi_{l_1}\nek \chi_{l_n}$ are considered as elements of $H^1(G)$.
Therefore the shuffle relations of Example \ref{shuffle relations examples} hold also in $H^2(G)$.

\section{Massey product criteria}
\label{section on Massey product criteria}
We can now translate Theorem \ref{partition criterion} to the context of Massey products.

\begin{thm}
\label{theorem on Massey products}
Let $G$ be a finitely presented pro-$p$ group.
Let $s,k_0,k_1\nek k_s$ be positive integers and set $n=k_0+k_1+\cdots+k_s$.
Let $U_0\nek U_s$ be nonzero $\dbF_p$-linear subspaces of $H^1(G)$ such that $H^1(G)=U_0\oplus\cdots\oplus U_s$.
Assume that:
\begin{enumerate}
\item[(a)]
$\grz(G)=n$.
\item[(b)]
One has $\langle\varphi_1\nek\varphi_n\rangle=0$ whenever $\varphi_1\nek\varphi_n\in H^1(G)$ and more than $k_j$ of them belong to $U_j$ for some $1\leq j\leq s$.
\item[(c)]
The image of the $n$-fold Massey product map 
\[
\langle\cdot\nek\cdot\rangle\colon U_0^{k_0}\times\cdots\times U_s^{k_s}\to H^2(G)
\]
generates $H^2(G)$.
\end{enumerate}
Then there exists a minimal presentation $G=F/R$ which is mild with respect to the initial form map $I\colon R\to \dbF_p\langle X\rangle$, $I(r)=\sum_{w\in X^n}\eps_w(r)w$.
\end{thm}
\begin{proof}
Take a minimal presentation $G=F/R$.
Denote $d=\dim_{\dbF_p}H^1(G)$ and identify $H^1(G)=H^1(F)$ via the inflation isomorphism.
Let $\chi_1\nek\chi_d$ be an $\dbF_p$-linear basis of $H^1(F)$ which is the union of bases of $U_0\nek U_s$.

Take $x_1\nek x_d\in F$ such that $\chi_j(x_i)=1$, if $i=j$, and $\chi_j(x_i)=0$ if $i\neq j$.
Then $X=\{x_1\nek x_d\}$ forms a basis of the free pro-$p$ group $F$.
We partition $X=Y_0\discup Y_1\discup\cdots\discup Y_s$, where for each $j=0,1\nek s$ the homomorphisms  $\chi_l$ with $x_l\in Y_j$ form a linear basis of $U_j$.  
For the standard pro-$p$ Magnus homomorphism $\Lam\colon F\to\dbF_p\langle\langle X\rangle\rangle^{\times,1}$, $x_i\mapsto 1+x_i$, we have $\eps_{(x_i)}=\chi_i$.

We now apply Theorem \ref{partition criterion} with $C=X^n$.
By (a) and Lemma \ref{Zassenhaus invariant and compatibility}(b), every word $w\in X^n$ is compatible with the presentation.

To verify condition (i) of the theorem, take $w=(x_{l(1)}\cdots x_{l(n)})\in X^n$ with more than $k_j$ letters from $Y_j$ for some $1\leq j\leq s$.
Thus more than $k_j$ of the homomorphisms $\chi_{l(1)}\nek\chi_{l(n)}$ belong to $U_j$.
By (\ref{Massey elements and Magnus coefficients}) and (b),
\[
\alp_w=\langle\eps_{(x_{l(1)})}\nek\eps_{(x_{l(n)})}\rangle=\langle\chi_{l(1)}\nek \chi_{l(n)}\rangle=0,
\]
as required.

Similarly, and in view of the multi-linearity of the Massey product, assumption (c) gives condition (ii) of Theorem \ref{partition criterion}.
\end{proof}

\begin{examples}
\rm
(1) \quad
The case $s=1$ of Theorem \ref{theorem on Massey products} is \cite{Gartner15}*{Th.\ 4.9}.

\medskip

(2) \quad
The case $s=1$, $k_1=k_2=1$, $n=2$, was earlier proved by Schmidt in \cite{Schmidt10}*{Th.\ 6.2} (see also \cite{Schmidt07}*{Th.\ 5.5}) in the case $p>2$.
A proof covering also the case $p=2$ was given in \cite{Gartner15}*{Remark 4.10(i)}.
Recall that the $2$-fold Massey product is just the cup product (see \S\ref{subsection on Massey products}), so conditions (b) and (c) of Theorem \ref{theorem on Massey products} can be restated in this case as follows:
\begin{enumerate}
\item[(a')]
The cup product map $\cup\colon U_1\tensor U_1\to H^2(G)$ is zero.
\item[(b')]
The cup product map $\cup\colon U_0\tensor U_1\to H^2(G)$ is surjective;
\end{enumerate}
Since always $\grz(G)\geq2$, condition (a) of the theorem follows from (b'), as $H^2(G)\neq0$.

\medskip

(3) \quad
The case $s=1$, $k_0=1$, $k_1=2$, $n=3$ was earlier proved in \cite{Gartner15}*{Remark 4.10(ii)}.
Furthermore, take a linear base $\chi_1\nek\chi_c$ of $U_0$ and a linear base $\chi_{c+1}\nek\chi_d$ of $U_1$, where $1\leq c<d$.
We claim that then conditions (a)--(c) of Theorem \ref{theorem on Massey products} can be restated as:
\begin{enumerate}
\item[(a')]
$\chi_i\cup\chi_j=0$ for all $1\leq i\leq j\leq d$;
\item[(b')]
$\langle \chi_i,\chi_j,\chi_l\rangle=0$ for all $c<i,j,l\leq d$;
\item[(c')]
The 3-fold Massey products $\langle \chi_i,\chi_j,\chi_l\rangle$, where $1\leq i\leq c$ and $c<j,l\leq d$, generate $H^2(G)$.
\end{enumerate}
Indeed, (a) and (a') are equivalent since the $2$-fold Massey product is the cup product.
The equivalence of (b) and (b') follows from the multi-linearity of the 3-fold Massey product (as $\grz(G)=3$).
Finally, the equivalence of (c) and (c') follows from the shuffle relations in Examples \ref{shuffle relations examples}(1)(4).
In view of Example \ref{small Lyndon words}, the restrictions on $i,j,l$ in (c') amount to restricting to Lyndon words. 

\medskip

(4) \quad
Similarly, the case $s=1$, $k_0=2$, $k_1=1$, $n=3$ was also noted in  \cite{Gartner15}*{Remark 4.10(ii)}.
Indeed, take $1\leq c<d$ as in the previous example.
Then assumptions (a)--(c) can be restated as:
\begin{enumerate}
\item[(a')]
$\chi_i\cup\chi_j=0$ for all $1\leq i\leq j\leq d$;
\item[(b')]
$\langle \chi_i,\chi_j,\chi_l\rangle=0$ whenever at least two of $i,j,l$ are $>c$;
\item[(c')]
The 3-fold Massey products $\langle \chi_i,\chi_j,\chi_l\rangle$, where $1\leq i,j\leq c$ and $c<l\leq d$, generate $H^2(G)$.
\end{enumerate}
Here we use again the shuffle relations of Examples \ref{shuffle relations examples}(1)(4).
Likewise, the restrictions on $i,j,l$ in (c') amount to restricting to Lyndon words with respect to the total order $x_1>\cdots>x_d$.

\medskip

(5) \quad
Higher values of $n$ can be handled similarly.
For instance, let $s=2$, $k_0=k_1=1$, $k_2=2$, $n=4$, and let $1\leq c<e<d$.
Take a linear basis $\chi_1\nek\chi_d$ of $H^1(G)=H^1(F)$.
We obtain that $G$ is mild provided that the following conditions hold:
\begin{enumerate}
\item[(a)]
$\chi_i\cup\chi_j=0$ for all $i,j$;
\item[(b)]
$\langle\chi_i,\chi_j,\chi_l\rangle=0$ for all $i,j,l$;
\item[(c)]
$\langle\chi_i,\chi_j,\chi_l,\chi_q\rangle=0$ whenever at least two of $i,j,l,q$ are $>c$ and $\leq e$;
\item[(d)]
$\langle\chi_i,\chi_j,\chi_l,\chi_q\rangle=0$ whenever at least theree of $i,j,l,q$ are $>e$;
\item[(e)]
The 4-fold Massey products $\langle \chi_i,\chi_j,\chi_l,\chi_q\rangle$ generate $H^2(G)$.
\end{enumerate}
By Remark (\ref{image of Lyndon words in cohomology}), in condition (b) (resp., condition (d)) we may further restrict ourselves to Massey products such that $(x_ix_jx_l)$ (resp., $(x_ix_jx_lx_q)$) is a Lyndon word with respect to some fixed total order on $X$.

\medskip

(6) \quad
Similarly, let $s=1$, $k_0=k_1=2$, $n=4$, and let $1\leq c<d$.
Take a linear basis $\chi_1\nek\chi_d$ of $H^1(G)=H^1(F)$.
We obtain that $G$ is mild provided that the following conditions hold:
\begin{enumerate}
\item[(a)]
$\chi_i\cup\chi_j=0$ for all $1\leq i\leq j\leq d$;
\item[(b)]
$\langle\chi_i,\chi_j,\chi_l\rangle=0$ for every $i,j,l$.
\item[(c)]
$\langle\chi_i,\chi_j,\chi_l,\chi_q\rangle=0$ whenever at least three of $i,j,l,q$ are $>c$;
\item[(d)]
The 4-fold Massey products $\langle \chi_i,\chi_j,\chi_l,\chi_q\rangle$ generate $H^2(G)$.
\end{enumerate}
As before, in condition (b) (resp., condition (d)) we may further restrict ourselves to Massey products such that $(x_ix_jx_l)$ (resp., $(x_ix_jx_lx_q)$) is a Lyndon word with respect to some fixed total order on $X$.
\end{examples}

\section{The non-singular circuit criterion}
\label{section on the non-singular circuit criteria}
In this section we show how Theorem \ref{main theorem on mildness} implies mildness criteria in the spirit of Labute's non-circular circuit criterion.

As before, let $F$ be a free pro-$p$ group on the basis $X=\{x_1\nek x_d\}$ and let $G=F/R$ be a finitely presented pro-$p$ group as in (\ref{minimal presentation}).
Suppose that the presentation is of \textsl{Koch type}, which means that $1\leq m\leq d$ and the defining relations $r_1\nek r_m$ in $R$ satisfy 
\[
r_j=x_j^{pa_j}\prod_{{k=1}\atop{k\neq j}}^d[x_j,x_k]^{a_{jk}}\pmod{F_{(3)}}
\]
with $0\leq a_j,a_{jk}\leq p-1$.

The following \textsl{non-singular circuit} criterion for mildness was proved by Labute \cite{Labute06} using Lie algebra techniques and under the additional assumptions that $p\neq2$, and the $a_{ji}$ which appear in condition (i) below are nonzero.
Note however that in \cite{Labute06} the mildness is with respect to the lower $p$-central filtration rather than the $p$-Zassenhaus filtration.
Below the coefficient maps $\eps_w$ are taken with respect to the standard Magnus homomorphism (Example \ref{standard Magnus homomorphism}).

\begin{thm}
\label{circuit d=m}
Assume that $d=m\geq4$ and that the following conditions hold:
\begin{enumerate}
\item[(i)]
$a_{12}a_{23}\cdots a_{d-1,d}a_{d1}\neq a_{21}a_{32}\cdots a_{d,d-1}a_{1d}$.
\item[(ii)]
$a_{jk}=0$ if $1\leq j,k\leq d$ are distinct and are both odd.
\end{enumerate}
Then the presentation $G=F/R$ is mild with respect to the initial form map $I(r)=\sum_{w\in X^2}\eps_w(r)w$.
\end{thm}

\begin{rem}
\label{d=m even}
\rm
Assumptions (i) and (ii) imply that $d=m$ is even.
\end{rem}

\begin{proof}
As before, we identify $H^1(G)=H^1(F)$ via inflation and take a linear basis $\chi_1\nek\chi_m$ f $H^1(F)$ which is dual to $X$, i.e., $\chi_j(x_i)$ is $1$ if $i=j$, and is $0$ if $i\neq j$.
Thus $\eps_{(x_i)}=\chi_i$ for every $i$.

Let $w=(x_ix_k)\in X^2$. 
By (\ref{Massey elements and Magnus coefficients}), $\alp_w=\eps_{(x_i)}\cup \eps_{(x_k)}=\chi_i\cup\chi_k$ (See \S\ref{subsection on Massey products}).
Further, $w$ is compatible with the presentation (Example \ref{compatibility for n=2}).
By (\ref{Magnus coefficients as transgression pairings}), for every $1\leq j\leq m$ we have
\[
\eps_w(r_j)=(\bar r_j,\alp_w)=(\bar r_j,\chi_i\cup\chi_k).
\]
Moreover, this is the exponent of $[x_i,x_k]$ in $r_j$ \cite{NeukirchSchmidtWingberg}*{Prop.\ 3.9.13}.
Therefore
\begin{equation}
\label{formula for Magnus coeffcients}\eps_w(r_j)=
\begin{cases}
0,&j\neq i,k,\\
a_{ik},&j=i,\\
-a_{ki},&j=k.
\end{cases}
\end{equation}

As $d$ is even (Remark \ref{d=m even}), we totally order $X$ by 
\[
x_2<x_4<\cdots<x_d<x_1<x_3<\cdots<x_{d-1}
\]
We order $X^*$ by the induced length-lexicographic order $\leq$, making $(X^*,\leq)$ an ordered monoid (Example \ref{length-lexicographic order}).

Now take $A=X^2$, and take $B$ to be the set of all words $(x_ix_k)$ with $i$ odd and $k$ even.
We verify conditions (a)--(c) of Theorem \ref{main theorem on mildness} with $\tau=1$.
%We have already noted the compatibility of $A$ compatible with the presentation.

Condition (a) is immediate.

For condition (b), take $w_1=(x_ix_k)\in B$ and $w_2=(x_{i'}x_{k'})\in X^2$ with $\alp_{w_1},\alp_{w_2}\neq0$ and $w_1\leq w_2$.
Then $i$ is odd and $k$ is even.
By the definition of the orders, $i'$ is also odd.
Furthermore, since the transgression pairing is perfect and $\alp_{w_2}\neq0$, there exists $j$ such that  $\eps_{w_2}(r_j)=(\bar r_j,\alp_{w_2})\neq0$ (See (\ref{Magnus coefficients as transgression pairings})).
By (\ref{formula for Magnus coeffcients}), either $j=i'$ and $a_{i'k'}\neq0$, or $j=k'$ and $a_{k'i'}\neq0$.
In both cases (ii) implies that $k'$ must be even.
Thus $w_2\in B$.

Finally we prove condition (c).
Denote $\tilde r=(r_1\nek r_m)$.
Let $B'$ be the subset of $B$ consisting of the words 
\[
\begin{split}
(x_1x_2)<(x_1x_d)&<(x_3x_2)<(x_3x_4)<\cdots\\
&<(x_{d-3}x_{d-4})<(x_{d-3}x_{d-2})<(x_{d-1}x_{d-2})<(x_{d-1}x_d).
\end{split}
\]
Using (\ref{formula for Magnus coeffcients}), we write down the matrix $M(\tilde r,B')=[\eps_w(r_j)]_{1\leq j\leq m,w\in B'}$ with the first two columns interchanged
(i.e., $(x_1x_d)$ appears before $(x_1x_2)$):
\begin{equation}
\label{matrix}
\begin{bmatrix}
a_{1d}&a_{12}&0&0&\cdots&0&0&0\\
0&-a_{21}&-a_{23}&0&\cdots&0&0&0\\
0&0&a_{32}&a_{34}&\cdots&0&0&0\\
\vdots&\vdots&\vdots&\ddots&\ddots&\vdots&\vdots&\vdots\\
0&0&0&0&\cdots&-a_{d-2,d-3}&-a_{d-2,d-1}&0\\
0&0&0&0&\cdots&0&a_{d-1,d-2}&a_{d-1,d}\\
-a_{d1}&0&0&0&\cdots&0&0&-a_{d,d-1}
\end{bmatrix}.
\end{equation}
By assumption (i), this matrix has a nonzero determinant, and therefore so does $M(\tilde r,B')$.
Since the full matrix $M(\tilde r,B)$ has $d=m$ rows, it therefore has rank $m$.
We conclude from Theorem \ref{conditions on the matrix M} that $\alp_w$, $w\in B$, generate $H^2(G)$, as required.
\end{proof}

We complement this by a mildness criterion for the case $m<d$.

\begin{thm}
\label{circuit m<d}
Assume that $1\leq m<d$ and the following conditions hold:
\begin{enumerate}
\item[(i)]
$a_{12},a_{23}\nek a_{m,m+1}\neq 0$.
\item[(ii)]
$a_{jk}=0$ if $1\leq j\leq m$ and $1\leq k\leq d$ are distinct and are both odd.
\end{enumerate}
Then the presentation $G=F/R$ is mild with respect to the initial term map $I(r)=\sum_{w\in X^2}\eps_w(r)w$.
\end{thm}
\begin{proof}
\null
The argument is similar to the previous one, in particular, the formula (\ref{formula for Magnus coeffcients}) holds.
We totally order $X$ by 
\[
x_2<x_4<\cdots<x_d<x_1<x_3<\cdots<x_{d-1}
\]
for $d$ even, and by
\[
x_2<x_4<\cdots<x_{d-1}<x_1<x_3<\cdots<x_d
\]
for $d$ odd.
We order $X^*$ by the induced length-lexicographic order $\leq$, making $(X^*,\leq)$ an ordered monoid.

Verifying conditions (a) and (b) of Theorem \ref{main theorem on mildness} is similar to the previous proof.
To verify condition (c), we take the subset $B'$ of $B$ to be
\[
\begin{split}
(x_1x_2)<(x_3x_2)&<(x_3x_4)<\cdots\\
&<(x_{d-3}x_{d-4})<(x_{d-3}x_{d-2})<(x_{d-1}x_{d-2})<(x_{d-1}x_d),
\end{split}
\]
and consider the $m\times d$-submatrix of (\ref{matrix}) consisting of its $m$ upper rows.
Then the $m\times m$-submatrix of this matrix consisting of rows $2,3\nek m+1$ has determinant $\pm a_{12}\cdots a_{m,m+1}\neq0$, so the rank is $m$.
\end{proof}

\begin{rem}
\rm
In Theorem \ref{circuit d=m} and Theorem \ref{circuit m<d} we may replace in condition (ii) the assumption that $a_{jk}=0$ when $j,k$ are both odd by the condition that $a_{jk}=0$ when $j,k$ are both even, and obtain analogous results.
Indeed, in the proofs we need to order $X$ so that the $x_k$ with $k$ odd are smaller than those with $k$ even, and argue similarly.
\end{rem}

\section{Pro-$p$ RAAGs}
\label{section on pro-p RAAGs}
Let $(X,E)$ be an unoriented finite graph with no loops.
Thus $X=\{x_1\nek x_d\}$ is a finite nonempty set of $d$ vertices and $E$ is a set of unordered pairs $\{x_i,x_j\}$, $i\neq j$, from $X$.
Let $F$ be the free pro-$p$ group on basis $X$ and let $R$ be its closed normal subgroup generated by the commutators $[x_i,x_k]$ with $\{x_i,x_k\}\in E$ (note that the order of $x_i$, $x_k$ does not matter).
Then $G=F/R$ is the \textsl{pro-$p$ right-angled Artin group} (abbreviated: RAAG) associated with $(X,E)$.

Quadrelli, Snopce and Vannacci \cite{QuadrelliSnopceVannacci22}*{Th.\ F}, building on work by Weigel \cite{Weigel15}, proved that a pro-$p$ RAAG $G$ is mild if and only if the underlying graph $(X,E)$ is \textsl{triangle-free}, i.e., it does not contain any $3$-clique.
We now show how Theorem \ref{main theorem on mildness} implies an important special case of this result, namely the case of graphs with chromatic number $2$, i.e., a bipartite graph.
Note that such graphs are necessarily triangle-free.

\begin{thm}
\label{theorem on RAAGs}
Let $(X,E)$ be a graph as above with chromatic number $2$.
Then the associated pro-$p$ RAAG  $G=F/R$  is mild with respect to the $p$-Zassenhaus initial form map.
\end{thm}
\begin{proof}
By assumption, there is a partition $X=Y_1\discup Y_2$ such that every edge in $E$ connects a vertex from $Y_1$ with a vertex from $Y_2$.

We choose a total order on $X$ in which the letters of $Y_1$ are larger than the letters of $Y_2$, and order $X^*$ length-lexicographically.
Take $n=2$, $A=X^2$ and let $B$ be the set of all words $(x_ix_k)$ such that $x_i\in Y_1$, $x_k\in Y_2$, and $\{x_i,x_k\}\in E$.
We work with the standard Magnus homomorphism (Example \ref{standard Magnus homomorphism}). 
We verify the conditions of Theorem \ref{main theorem on mildness}.

By Example \ref{compatibility for n=2}, the words in $A$ are compatible with the presentation.
It is straightforward to see that $B$ is combinatorially free.

For (b), take $w_1\in B$ and $w_2\in A$ with $\alp_{w_1},\alp_{w_2}\neq0$ and $w_1\leq w_2$.
Since the transgression pairing is perfect and by (\ref{Magnus coefficients as transgression pairings}), there exists a defining relation $r$ (corresponding to an edge) with $\eps_{w_2}(r)=(\bar r,\alp_{w_2})_2\neq0$.
We may write $r=[x_i,x_k]$ where $x_i\in Y_1$ and $x_k\in Y_2$.
Then
\begin{equation}
\label{Magnus expansion of commutator}
\Lam(r)=1+x_ix_k-x_kx_i+\cdots, 
\end{equation}
where the remaining terms are of degree $\geq3$.
Therefore $w_2$ is one of $(x_ix_k)$ or $w_2=(x_kx_i)$.
But by the choice of the order, $(x_kx_i)$ is strictly smaller than all the words in $B$.
Therefore $w_2=(x_ix_k)\in B$.

Finally, the matrix $M(\tilde r,B)=[\eps_w(r_j)]_{1\leq j\leq m,w\in B}$, where $r_1\nek r_m$ are the defining relations, is diagonal with $\pm1$ on the diagonal entries, by (\ref{Magnus expansion of commutator}).
Therefore it has rank $m$, so (c) also holds.
\end{proof}

\begin{exam}
\rm
Suppose that $d\geq4$ is even.
The ``line group" (resp., ``circle group") is defined by $G=F/R$, where $F$ is free on the basis $x_1\nek x_d$ and $R$ is generated by 
\[
[x_1,x_2],[x_2,x_3],[x_3,x_4],[x_4,x_5]\nek[x_{d-1},x_d],
\] 
resp.,
\[
[x_1,x_2],[x_2,x_3],[x_3,x_4],[x_4,x_5]\nek[x_{d-1},x_d],[x_d,x_1].
\] 
By Theorem \ref{theorem on RAAGs}, these groups are mild with respect to the Zassenhaus initial form map. 
For the circle group this was earlier shown in \cite{Gartner15}*{Example 2.10(i)}.
See also \cite{MinacPasiniQuadrelliTan22}*{\S5}.
\end{exam}

\end{document}